\documentclass[reqno,12pt]{amsart}
 \usepackage{amsmath}
 \usepackage{amsthm}
 \usepackage{amssymb}
 \usepackage{latexsym,longtable}
 \usepackage{graphicx}
 \usepackage{multicol}
 \usepackage{mathrsfs}
 \usepackage{young}
 \usepackage[vcentermath,enableskew,stdtext]{youngtab}
 \usepackage[left=1.1in,right=1.1in,top=1.1in,bottom=1.14in]{geometry}
 \usepackage[all]{xy}
    \SelectTips{cm}{10}     
    \everyxy={<2.5em,0em>:} 
\usepackage{epigraph}
 \usepackage{fancyhdr}      
 \usepackage{multicol}
 \usepackage{multirow}
 \usepackage{enumitem}
 \usepackage{bm}
 \usepackage{stmaryrd}
 \usepackage{etex}
 \usepackage{pgf}
\usepackage{tikz}
 \usepackage{float}
 \usepackage{caption 
}
 \usepackage{array}
 \usepackage{colortbl}
 \usepackage{mathtools}
 \restylefloat{figure}
 \usetikzlibrary{shapes}
 \usetikzlibrary{arrows}
 \usetikzlibrary{snakes}
\usepackage{url}

\usepackage{amsfonts,epsfig,color}
\makeatletter
\newcommand*{\centerfloat}{%
  \parindent \z@
  \leftskip \z@ \@plus 1fil \@minus \textwidth
  \rightskip\leftskip
  \parfillskip \z@skip}
\makeatother

\theoremstyle{plain}
\newtheorem{theorem}{Theorem}[section]
\newtheorem{lemma}[theorem]{Lemma}
\newtheorem{corollary}[theorem]{Corollary}
\newtheorem{proposition}[theorem]{Proposition}
\newtheorem{conjecture}[theorem]{Conjecture}
\newtheorem{problem}[theorem]{Problem}

\theoremstyle{definition}
\newtheorem{definition}[theorem]{Definition}

\newtheorem{remark}[theorem]{Remark}
\newtheorem{example}[theorem]{Example}

\numberwithin{equation}{section}

\newcommand{\ignore}[1]{}

\newcommand{\QQ}{\ensuremath{\mathbb{Q}}}

\newcommand{\be}{\begin{equation}}
\newcommand{\ee}{\end{equation}}

\newcommand{\crc}[1]{\ensuremath{\overline{#1}}\vphantom{\underline{\overline{#1}}}}

\newlength{\cellsize}
\newcommand\tableau[1]{
\vcenter{
\let\\=\cr
\baselineskip=-16000pt \lineskiplimit=16000pt \lineskip=0pt
\halign{&\tableaucell{##}\cr#1\crcr}}}

\newcommand{\tableaucell}[1]{{%
\def \arg{#1}\def \void{}%
\ifx \void \arg
\vbox to \cellsize{\vfil \hrule width \cellsize height 0pt}%
\else \unitlength=\cellsize
\begin{picture}(1,1)
\put(0,0){\makebox(1,1){$#1\vphantom{\crc{#1}}$}}
\put(0,0){\line(1,0){1}}
\put(0,1){\line(1,0){1}}
\put(0,0){\line(0,1){1}}
\put(1,0){\line(0,1){1}}
\end{picture}%
\fi}}

\newcommand\boldtableau[1]{
\vcenter{
\let\\=\cr
\baselineskip=-16000pt \lineskiplimit=16000pt \lineskip=0pt
\halign{&\boldtableaucell{##}\cr#1\crcr}}}

\newcommand{\boldtableaucell}[1]{{%
\def \arg{#1}\def \void{}%
\ifx \void \arg
\vbox to \cellsize{\vfil \hrule width \cellsize height 0pt}%
\else \unitlength=\cellsize
\begin{picture}(1,1)
\put(0,0){\makebox(1,1){$\mathbf{#1\vphantom{\crc{#1}}}$}}
\put(0,0){\line(1,0){1}}
\put(0,1){\line(1,0){1}}
\put(0,0){\line(0,1){1}}
\put(1,0){\line(0,1){1}}
\end{picture}%
\fi}}

\newcommand{\mybar}[1]{\crc{#1}}
\newcommand{\esup}[2]{\ensuremath{\mybar{e}_{#1}(#2)}}

\title{Rules of Three for commutation relations}


\begin{document}

\author{Jonah Blasiak}
\address{Department of Mathematics, Drexel University, Philadelphia, PA 19104}
\email{jblasiak@gmail.com}

\author{Sergey Fomin}
\address{Department of Mathematics, University of Michigan, Ann Arbor,
  MI 48109}
\email{fomin@umich.edu}

\thanks{Partially supported by NSF grants DMS-14071174 (J.~B.)
and DMS-1361789 (S.~F.).}

\date{August 16, 2016}

\dedicatory{Dedicated to Efim Zelmanov on his 60th birthday}


\begin{abstract}
We study the phenomenon in which commutation relations for
sequences of elements in a ring
are implied by similar relations
for subsequences involving
at most three indices at a time.
\end{abstract}%

\maketitle
\setlength{\cellsize}{2.2ex}

\setlength{\epigraphwidth}{270pt}
\epigraph{
          What I tell you three times is true. 
}{\ \\[-.05in]
\textit{The Hunting of
    the Snark}, by Lewis Carroll}


\section*{Introduction}

In this paper, we investigate the following surprisingly widespread
phenomenon which we call \emph{The Rule of Three}:
in order for a particular kind of commutation relation to hold for
subsequences of elements of a ring labeled by any subset of indices,
it is enough that these relations hold for subsets of size one, two, and three.

Here is a typical ``Rule of Three'' statement.
Let $g_1,\ldots, g_N, h_1,\ldots, h_N$ be invertible elements in an
associative ring. Then the following are equivalent (cf.\ Theorem~\ref{t half group ROT}):
\begin{itemize}
\item
for any subsequence of indices $1\le s_1< \cdots< s_m\le N$,
the element $g_{s_m}\cdots g_{s_1}$
commutes with both $h_{s_m}\cdots h_{s_1}$ and $h_{s_m}+\cdots
+h_{s_1}$;
\item
the above condition holds for all subsequences of length $m\le 3$.
\end{itemize}

We establish many results of this form, including
\begin{itemize}
\item Rules of Three for noncommutative elementary symmetric functions
 (Section~\ref{s main results 1});
\item Rules of Three for generating functions over rings
(Section~\ref{s generating functions over rings});
\item Rules of Three for sums and products (Section~\ref{sec:rot-invertible}).
\end{itemize}

Proofs are given in
Sections~\ref{sec:rot-groups}--\ref{sec:proof-super}.
For reference,
Theorems \ref{t pairs commute v0} and \ref{th:mrot-linear-uu} are proved in Section \ref{sec:proof:u=v};
Theorem~\ref{t half ROT} is proved in Section \ref{s proof of theorem t half ROT};
Theorems~\ref{t es commute uv 2}, \ref{t half group ROT}, \ref{t 2variable linear strengthened}, \ref{t half quadratic half group ROT}, and~\ref{t half group ROTv2} are proved in Section \ref{sec:main-proofs}; and
Theorem~\ref{t super} is proved in Section \ref{sec:proof-super}.


\section{Rules of Three for noncommutative symmetric functions}
\label{s main results 1}
Let
$\mathbf{u}=(u_1,\dots,u_N)$ be an ordered $N$-tuple of elements in a
ring~$R$.
(We informally view $u_1,\dots,u_N$ as ``noncommuting variables.'')
For an integer~$k$,
the \emph{noncommutative elementary symmetric function}
$e_k(\mathbf{u})\in R$ is defined by
\begin{align}
\label{eq:e_k(u)}
e_k(\mathbf{u})=\sum_{N \ge i_1 > i_2 > \cdots > i_k \ge 1} u_{i_1}u_{i_2}\cdots u_{i_k}\,.
\end{align}
(By convention,
$e_0(\mathbf{u})\!=\!1$ and $e_k(\mathbf{u})\! =\! 0$ if $k<0$ or $k>N$.)
More generally, for a subset $S\subset \{1,\dots,N\}$,
we denote
\begin{equation}\label{e ek def}
e_k(\mathbf{u}_S)=\sum_{\substack{i_1 > i_2 > \cdots > i_k
    \\ i_1,\dots,i_k \in S}}u_{i_1}u_{i_2}\cdots u_{i_k}\,.
\end{equation}
Again,
$e_0(\mathbf{u}_S)\!=\!1$, and $e_k(\mathbf{u}_S) \!=\! 0$ unless
$0\le k\le |S|$.
(Here $|S|$ is the cardinality of~$S$.)

\begin{theorem}[The Rule of Three for noncommutative elementary
      symmetric functions]
\label{t es commute uv 2}
Let $\mathbf{u}=(u_1,\dots,u_N)$
and $\mathbf{v}=(v_1,\dots,v_N)$ be ordered $N$-tuples of elements in
a ring~$R$.
Then the following conditions are equivalent:
\begin{itemize}
\item
the noncommutative elementary symmetric functions
$e_k(\mathbf{u}_S)$ and $e_\ell(\mathbf{v}_S)$ commute with each other,
for any integers $k$ and~$\ell$ and any subset $S \subset\{1,\dots,N\}$:
\begin{equation}
\label{eq:e-commute-uv}
e_k(\mathbf{u}_S)\,e_\ell(\mathbf{v}_S)=e_\ell(\mathbf{v}_S)\,e_k(\mathbf{u}_S);
\end{equation}
\item
the commutation relation \eqref{eq:e-commute-uv} holds
for $|S|\le 3$ and any $k, \ell$;
\item
the commutation relation \eqref{eq:e-commute-uv} holds
for $|S|\le 3$ and $k\ell\le 3$.
\end{itemize}
\end{theorem}

\begin{remark}
Explicitly, Theorem~\ref{t es commute uv 2} asserts that the
commutation relations~\eqref{eq:e-commute-uv}
hold for all $k$ and~$\ell$ and all subsets $S \subset \{1,\dots,N\}$
if and only if the following relations hold:
\begin{align}
&e_1(\mathbf{u}_S)e_1(\mathbf{v}_S) =
  e_1(\mathbf{v}_S)e_1(\mathbf{u}_S) \qquad \text{for }\, 1 \le |S|
  \le 3, \label{e uv 12vars AB}\\
&e_2(\mathbf{u}_S)e_1(\mathbf{v}_S) =
  e_1(\mathbf{v}_S)e_2(\mathbf{u}_S) \qquad \text{for }\, 2 \le |S|
  \le 3, \label{e uuv 23vars AB}\\
&e_1(\mathbf{u}_S)e_2(\mathbf{v}_S) =
  e_2(\mathbf{v}_S)e_1(\mathbf{u}_S) \qquad \text{for }\, 2 \le |S|
  \le 3, \label{e uvv 23vars AB}\\
&e_3(\mathbf{u}_S)e_1(\mathbf{v}_S) =
  e_1(\mathbf{v}_S)e_3(\mathbf{u}_S) \qquad \text{for }\, |S| =
  3, \label{e uuuv 23vars AB}\\
&e_1(\mathbf{u}_S)e_3(\mathbf{v}_S) =
  e_3(\mathbf{v}_S)e_1(\mathbf{u}_S) \qquad \text{for }\, |S| =
  3. \label{e uvvv 23vars AB}
\end{align}
(Actually, it suffices to require \eqref{e uv 12vars AB} for $|S|\le 2$,
but this is not so important.)
It is rather miraculous that 
\eqref{e uv 12vars AB}--\eqref{e uvvv 23vars AB} imply
the relations
\begin{align}
&e_2(\mathbf{u}_S)e_2(\mathbf{v}_S) = e_2(\mathbf{v}_S)e_2(\mathbf{u}_S)
\qquad \text{for }\, 2 \le |S|
  \le 3, \\
&e_2(\mathbf{u}_S)e_3(\mathbf{v}_S) =
  e_3(\mathbf{v}_S)e_2(\mathbf{u}_S) \qquad \text{for }\, |S| =
  3, \\
&e_3(\mathbf{u}_S)e_2(\mathbf{v}_S) =
  e_2(\mathbf{v}_S)e_3(\mathbf{u}_S) \qquad \text{for }\, |S| =
  3, \\
&e_3(\mathbf{u}_S)e_3(\mathbf{v}_S) =
  e_3(\mathbf{v}_S)e_3(\mathbf{u}_S) \qquad \text{for }\, |S| =
  3,
\end{align}
in addition to all relations \eqref{eq:e-commute-uv} for $|S|\ge 4$.
\end{remark}


In the case of a single $N$-tuple of ``noncommuting variables''
$u_1=v_1, \dots, u_N=v_N\,$, we obtain the following corollary.

\begin{corollary}[{\cite{Kirillov-Notes}, \cite{BF}}]
\label{cor intro ed commute}
Let $R$ be a ring, and let $\mathbf{u}\!=\!(u_1,\dots,u_N)$
be an ordered $N$-tuple of elements of~$R$.
Then the following are equivalent:
\begin{itemize}
\item
the noncommutative elementary symmetric functions
$e_k(\mathbf{u}_S)$ and $e_\ell(\mathbf{u}_S)$ commute with each other,
for any integers $k$ and~$\ell$ and any subset $S
\subset\{1,\dots,N\}$:
\begin{equation}
\label{eq:e-commute-u}
e_k(\mathbf{u}_S)e_\ell(\mathbf{u}_S)=e_\ell(\mathbf{u}_S)e_k(\mathbf{u}_S);
\end{equation}
\item
the following special cases of~\eqref{eq:e-commute-u} hold:
\begin{align}
&e_1(\mathbf{u}_S)e_2(\mathbf{u}_S) =
  e_2(\mathbf{u}_S)e_1(\mathbf{u}_S) \qquad \text{ for } 2 \le |S| \le
  3, \label{e u12 23vars}\\
&e_1(\mathbf{u}_S)e_3(\mathbf{u}_S) =
  e_3(\mathbf{u}_S)e_1(\mathbf{u}_S) \qquad \text{ for } |S| =
  3. \label{e u13 23vars}
\end{align}
\end{itemize}
\end{corollary}

\begin{proof}
This is a special case of Theorem~\ref{t es commute uv 2}.
Note that when $\mathbf{u}=\mathbf{v}$, the condition \eqref{e uv 12vars AB}
is trivial,
whereas \eqref{e uuv 23vars AB}--\eqref{e uvvv 23vars AB} become
\eqref{e u12 23vars}--\eqref{e u13 23vars}.
\end{proof}

\begin{remark}
\label{r similar results elem}
Corollary~\ref{cor intro ed commute} is equivalent to a result by
A.~N.~Kirillov~\cite[Theorem~2.24]
{Kirillov-Notes}; the above version appeared in our previous paper \cite[Theorem~2.3]{BF}.
It generalizes similar results obtained in \cite{BD0graph, FG, LamRibbon, LS, NS}.
See \cite[Remark 2.2]{BF} for additional discussion.
\end{remark}

\begin{remark}\label{r nc schur}
Corollary~\ref{cor intro ed commute} and similar results
serve as the starting
point for the theory~of \emph{noncommutative Schur functions},
which aims to produce positive combinatorial formulae for Schur
expansions of various classes of symmetric functions.
This theory originated in~\cite{FG}, building off the
work of Lascoux and Sch\"utzenberger on the plactic algebra~\cite{LS, Sch}.
It was \linebreak[3]
later adapted to study LLT polynomials~\cite{LamRibbon}
and $k$-Schur functions~\cite{LamAffineStanley};
other variations appeared in
\cite{BenedettiBergeron, Kirillov-Notes, NS}.
%
Further recent work includes the papers~\cite{BD0graph, BLamLLT, BF},
which advance the theory to encompass Lam's work \cite{LamRibbon} and
incorporate ideas of Assaf~\cite{SamiOct13}. \linebreak[3]
One of the main outcomes of this approach is a proof of Haglund's
conjecture on 3-column Macdonald polynomials~\cite{BLamLLT}.
\end{remark}

\vspace{-2mm}

Theorem~\ref{t es commute uv 2} can be generalized to the setting of
``noncommutative supersymmetric polynomials.''
Let us fix an arbitrary partition of the ordered alphabet $\{1 < \cdots < N\}$
into \emph{unbarred} and \emph{barred indices}.
The \emph{noncommutative super elementary symmetric function}
$\mybar{e}_k({\mathbf{u}})$
is defined by the following variation of~\eqref{eq:e_k(u)}:
\begin{align}
\label{eq:ebar_k(u)}
\mybar{e}_k(\mathbf{u})=
\sum_{\substack{N \ge i_1 \ge i_2 \ge \cdots \ge i_k
    \ge 1\\ \text{$i_j$ unbarred $\Rightarrow i_j>i_{j+1}$}}}
u_{i_1}\cdots u_{i_k}\,.
\end{align}
We similarly define the elements $\mybar{e}_k({\mathbf{u}}_S)$
associated to sub-alphabets $S\subset\{1 < \cdots < N\}$.

\begin{theorem}
\label{t super}
Let $R$ be a ring, and let
$\mathbf{u}\!=\!(u_1,\dots,u_N)$ and  $\mathbf{v}\!=\!(v_1,\dots,v_N)$
be ordered $N$-tuples of elements of~$R$.
Then the following are equivalent:
\begin{itemize}
\item
$\esup{k}{\mathbf{u}_S}$ and $\esup{\ell}{\mathbf{v}_S}$ commute, 
for any 
$k$ and~$\ell$ and any subset $S
\subset\{1,\dots,N\}$:
\begin{equation}
\label{eq:e-commute-u-super}
\esup{k}{\mathbf{u}_S}\esup{\ell}{\mathbf{v}_S}=\esup{\ell}{\mathbf{u}_S}\esup{k}{\mathbf{v}_S};
\end{equation}
\item
the following special cases of~\eqref{eq:e-commute-u-super} hold:
\begin{align}
&\esup{k}{\mathbf{u}_S}\esup{1}{\mathbf{v}_S} =
  \esup{1}{\mathbf{u}_S}\esup{k}{\mathbf{v}_S} \qquad
\text{ for  $|S| \le  3$ and  $k \ge 1$};  \label{e 1d super}\\
&\esup{1}{\mathbf{u}_S}\esup{\ell}{\mathbf{v}_S} =
  \esup{\ell}{\mathbf{u}_S}\esup{1}{\mathbf{v}_S} \qquad
\text{ for  $|S| \le  3$ and  $\ell \ge 1$}.  \label{e 1d super2}
\end{align}
\end{itemize}
\end{theorem}

We discuss the broader context for Theorem~\ref{t super}
in Remark~\ref{rem:rot-super} below.

\pagebreak[3]

\section{Rules of Three for generating functions over rings}
\label{s generating functions over rings}
Commutation relations for noncommutative elementary symmetric
functions can be reformulated as multiplicative identities for certain
elements of a polynomial ring in two (central) variables with coefficients in~$R$.
This leads to an
alternative perspective on The Rules of Three, which we discuss next.

In the rest of this paper, we repeatedly make use of the following
convenient notation.
Let $g_1,\dots,g_N$ be elements of a ring (or a monoid),
and let $S=\{s_1\!<\!\cdots\!<\!s_m\}\!\subset\!\{1,\dots,N\}$ be a subset of
indices.
We then denote
\begin{equation}
\label{eq:g_S}
g_S = g_{s_m}\cdots g_{s_1}\,.
\end{equation}
We similarly use the shorthand $h_S= h_{s_m}\cdots h_{s_1}$, etc.

\begin{corollary}
\label{cor:rule-of-3-uv-linear}
Let $R[x,y]$ be the ring of polynomials in the
formal variables $x$ and~$y$ with coefficients in a ring~$R$.
(Here $x$ and~$y$ commute with each other and with any $z\in R$.)
Let $u_1,\dots,u_N, v_1,\dots,v_N\in R$.
For $i=1,\dots,N$, set
\begin{align}
g_i = 1+x u_i \in R[x,y], \label{eq:gi-linear} \\
h_i = 1+y v_i \in R[x,y]. \label{eq:hi-linear}
\end{align}
Then the following are equivalent:
 \begin{itemize}
\item $g_S h_S = h_S g_S$ for all subsets $S\subset\{1,\dots,N\}$;
\item $g_S h_S = h_S g_S$ for all subsets $S$ of cardinality $1$, $2$,
  and~$3$.
\end{itemize}
\end{corollary}

\begin{proof}
We observe that \eqref{eq:gi-linear}--\eqref{eq:hi-linear} imply
\begin{alignat}{2}
g_S 
&= (1+x u_{s_m})\cdots (1+x u_{s_1})&&=\textstyle\sum_k
x^k \,e_k(\mathbf{u}_S), \label{eq:gS-linear} \\
h_S 
&= (1+y v_{s_m})\cdots (1+y v_{s_1})&&=\textstyle\sum_\ell
y^\ell \,e_\ell(\mathbf{v}_S). \label{eq:hS-linear}
\end{alignat}
Thus the property~\eqref{eq:e-commute-uv}
(that is, each $e_k(\mathbf{u}_S)$ commutes with each $e_\ell(\mathbf{v}_S)$)
is equivalent to saying that $g_S$ commutes with~$h_S$.
The corollary is now immediate from Theorem~\ref{t es commute uv 2}.
\end{proof}


Given the multiplicative form of the conditions $g_S h_S = h_S g_S$ in
Corollary~\ref{cor:rule-of-3-uv-linear},
it~is tempting to seek group-theoretic generalizations of the latter,
with the factors $g_i$ and~$h_i$ drawn from some (reasonably general)
group.
Unfortunately,
the purely 
group-theoretic extension of Corollary~\ref{cor:rule-of-3-uv-linear}
is false:
the relation
$g_4g_3g_2g_1h_4h_3h_2h_1 = h_4h_3h_2h_1g_4g_3g_2g_1$  
does not hold
in the group with presentation given by generators $g_1, \ldots, g_4,
h_1, \ldots, h_4$ and relations $g_Sh_S = h_Sg_S$ for  $|S| \le 3$.
(This 
follows from the fact that
replacing \eqref{eq:hi-linear} with $h_i=1+x v_i$
transforms Corollary~\ref{cor:rule-of-3-uv-linear} into a false
statement, cf.\ Example~\ref{ex:256patterns}.)

Consequently one has to introduce some
(likely nontrivial) assumptions on  the group~$G$
and/or the elements $g_i, h_i$.
Two results of this kind are stated in Section~\ref{sec:rot-groups}.
A fundamental question remains (see also Problem \ref{problem:rot:A[[x,y]]}):
\begin{problem}
\label{problem:group theoretic ROT}
Find a group-theoretic Rule of Three strong enough
to directly imply Corollary~\ref{cor:rule-of-3-uv-linear}
(or better yet, Conjecture~\ref{cj pairs commute uv} below).
\end{problem}

\pagebreak[3]

From the standpoint of potential applications,
the most important setting for ``multiplicative rules of three''
\emph{\`a~la} Corollary~\ref{cor:rule-of-3-uv-linear}
is the one where the factors $g_i, h_i$ are formal power series
in $xu_i$ and~$yv_i$, respectively.
Extensive computational evidence suggests that in this setting, the
Rule of Three always holds:

\begin{conjecture}
\label{cj pairs commute uv}
Let $R[[x,y]]$ be the ring of formal power series in the variables $x$
and~$y$ with coefficients in a $\QQ$-algebra~$R$.
Let $u_1,\dots,u_N, v_1,\dots,v_N\in R$, and
assume that $g_1, \ldots, g_N$, $h_1, \ldots, h_N \in R[[x,y]]$ are power series of the form
\begin{alignat}{5}
g_i &= 1+\alpha_{i1} xu_i &&+\alpha_{i2} (xu_i)^2 &&+\alpha_{i3}
(xu_i)^3&&+\cdots
\qquad && (\alpha_{ik}\in\QQ), \label{eq:g_i=g_i(xu_i)} \\
h_i &= 1+\beta_{i1} yv_i &&+\beta_{i2} (yv_i)^2 &&+\beta_{i3}
(yv_i)^3&&+\cdots
\qquad && (\beta_{ik}\in\QQ), \label{eq:h_i=h_i(yv_i)}
\end{alignat}
where for every~$i$, either $\alpha_{i1}$ or~$\beta_{i1}$ is nonzero.
Then the following are equivalent:
 \begin{itemize}
\item $g_S h_S = h_S g_S$ for all subsets $S\subset\{1,\dots,N\}$;
\item $g_S h_S = h_S g_S$ for all subsets $S$ of cardinality $1$, $2$,
  and~$3$.
\end{itemize}
\end{conjecture}

While Conjecture~\ref{cj pairs commute uv} remains open,
we were able to prove it in several important cases.
Three such results appear below and two more appear in Corollaries \ref{c half group ROT}
and \ref{c cj pairs commute uv version3}.

First, we obtain the following generalization of Corollary~\ref{cor:rule-of-3-uv-linear}.

\begin{theorem}
\label{t half group ROT corollary mini}
Conjecture~\ref{cj pairs commute uv} holds for $h_i = 1+yv_i$ (and
any~$g_i$ as in \eqref{eq:g_i=g_i(xu_i)}).
\end{theorem}

Theorem~\ref{t half group ROT corollary mini} is a special case of a more general result,
see Corollary~\ref{c half group ROT}.

We also settle Conjecture~\ref{cj pairs commute uv}
in the case of a single set of noncommuting variables:

\begin{theorem} 
\label{t pairs commute v0}
Let $R[[x,y]]$ be the ring of formal power series in $x$
and~$y$ with coefficients in a $\QQ$-algebra~$R$.
Let $u_1,\dots,u_N\in R$.
Let $g_1, \ldots, g_N, h_1, \ldots, h_N \in R[[x,y]]$ be of the form
\begin{alignat}{5} \label{eq:gi-uu}
g_i &= 1+\alpha_{i1} xu_i &&+\alpha_{i2} (xu_i)^2 &&+\alpha_{i3}
(xu_i)^3&&+\cdots
\qquad && (\alpha_{ik}\in\QQ), \\
h_i &= 1+\beta_{i1} yu_i &&+\beta_{i2} (yu_i)^2 &&+\beta_{i3}
(yu_i)^3&&+\cdots
\qquad && (\beta_{ik}\in\QQ),\label{eq:hi-uu}
\end{alignat}
where for every~$i$, either $\alpha_{i1}$ or~$\beta_{i1}$ is nonzero.
Then the following are equivalent:
 \begin{itemize}
\item $g_S h_S = h_S g_S$ for all subsets $S\subset\{1,\dots,N\}$;
\item $g_S h_S = h_S g_S$ for all subsets $S$ of cardinality $2$
  and~$3$.
\end{itemize}
\end{theorem}

%
%

Yet another case of Conjecture~\ref{cj pairs commute uv} follows from Theorem~\ref{t super}:
\begin{corollary}
\label{c super}
Conjecture~\ref{cj pairs commute uv} holds provided for each~$i$,
one of the following two options is chosen:
\begin{itemize}
\item $g_i = 1+x u_i$ and $h_i = 1+y v_i$; or
\item $g_i=(1-xu_i)^{-1}$ and $h_i=(1-yv_i)^{-1}$.
\end{itemize}
\end{corollary}

Theorem~\ref{t super} is stronger than Corollary~\ref{c super} since
the latter requires the relations \eqref{eq:e-commute-u-super} for
$|S| \le 3$ and any  $k, \ell$ whereas the former only needs the
instances with $k =1$ or  $\ell=1$.

\begin{remark}
\label{rem:rot-super}
Theorem~\ref{t pairs commute v0} is a far-reaching generalization of
Corollary~\ref{cor intro ed commute}, which already
demonstrated its importance in initiating new developments in the theory of noncommutative Schur functions
 (cf. Remark~\ref{r nc schur}).
It remains to be seen whether Theorem~\ref{t pairs commute v0}
for general power series \eqref{eq:gi-uu}--\eqref{eq:hi-uu}
can spawn new versions of this theory. 

One case (beyond the basic choice
\hbox{\eqref{eq:gi-linear}--\eqref{eq:hi-linear}})
where some progress has been made is the setting of noncommutative
super symmetric functions (cf.\ Theorem \ref{t super}).
The ring defined by the relations \eqref{e u12 23vars}--\eqref{e u13 23vars}
has  many quotients with rich combinatorial
structure (the plactic algebra, nilCoxeter algebra, and more, see \cite{BD0graph, BF, FG, LamRibbon});
the ring defined by the relations \eqref{e 1d super}--\eqref{e 1d super2}
has many interesting quotients as well, some of them similar to the plactic algebra.
The recent paper \cite{BLKronecker}
studies some of these quotients and develops an accompanying theory of
noncommutative super Schur functions.
The main application (recovering results
of~\cite{BHook, Ricky}) is a positive combinatorial rule for the Kronecker coefficients  where one
of the shapes is a hook.
\end{remark}

We next discuss some of the subtleties involved in
generalizing the above results.
To~facilitate this discussion, we introduce the following concept.

\begin{definition}
\label{def:mult-rot}
Let $\mathcal{A}=\QQ\langle u_1, \ldots, u_M, v_1, \ldots, v_M\rangle$
be the free associative $\QQ$-algebra with generators $u_1,\ldots, u_M, v_1,\ldots, v_M$.
Let $g_1, \ldots, g_N, h_1, \ldots, h_N$ be elements of~$\mathcal{A}[[x,y]]$,
i.e., some formal power series in $x$ and~$y$ with coefficients in~$\mathcal{A}$.
We say that the \emph{Multiplicative Rule of Three holds} for 
  $g_1,\ldots,g_N,h_1,\ldots,h_N$ if for any quotient ring
$R=\mathcal{A}/I$, the following are equivalent:
\begin{itemize}
\item $g_S h_S \equiv h_S g_S \bmod I$\ \ for all subsets $S\subset\{1,\dots,N\}$;
\item $g_S h_S \equiv h_S g_S \bmod I$\ \ for all subsets $S$ of cardinality $\le 3$.
\end{itemize}
\end{definition}

For example, the Multiplicative Rule of Three holds in the following cases:
\begin{itemize}
\item
$h_i = 1+yv_i$ and
any~$g_i$ as in \eqref{eq:g_i=g_i(xu_i)} (by Theorem~\ref{t half group ROT corollary mini});
\item
$g_i, h_i$ are given by \eqref{eq:gi-uu}--\eqref{eq:hi-uu}, with $\alpha_{i1}, \beta_{i1}$ not both 0
(by Theorem~\ref{t pairs commute v0}).
\end{itemize}
Conjecture \ref{cj pairs commute uv} asserts that the
Multiplicative Rule of Three holds for
$g_i, h_i$ given by \eqref{eq:g_i=g_i(xu_i)}--\eqref{eq:h_i=h_i(yv_i)},
with  $\alpha_{i1}, \beta_{i1}$ not both 0.

\begin{problem}
\label{problem:rot:A[[x,y]]}
Find the most general setting
(i.e., the weakest restrictions on the expressions $g_i, h_i$)
for which the Multiplicative Rule of Three~holds.
\end{problem}

\begin{example}
The Multiplicative Rule of Three fails for
\begin{alignat}{2}
&g_4 = (1+x u_8)(1+x u_7), \qquad && h_4 = (1+y u_8)(1+y u_7), \label{eq:paired-factors1} \\
&g_3 = (1+x u_6)(1+x u_5), && h_3 = (1+y u_6)(1+y u_5), \label{eq:paired-factors2} \\
&g_2 = (1+x u_4)(1+x u_3), && h_2 = (1+y u_4)(1+y u_3), \label{eq:paired-factors3} \\
&g_1 = (1+x u_2)(1+xu_1),  && h_1 = (1+y u_2)(1+y u_1). \label{eq:paired-factors4}
\end{alignat}
In other words, the relations on $8$~elements $u_1,\dots,u_8$
of a ring~$R$ resulting from the conditions
$g_S h_S = h_S g_S$ for all subsets $S$ of cardinality~$\le 3$,
with the $g_i$ and the~$h_i$ given by \eqref{eq:paired-factors1}--\eqref{eq:paired-factors4},
do not imply the relation $g_S h_S = h_S g_S$ for $S=\{1,2,3,4\}$.
This was shown using a noncommutative Gr\"obner basis calculation in \texttt{Magma}~\cite{Magma}.
\end{example}


\begin{example}
\label{ex:256patterns}
For an $8$-letter word $\mathbf{z}=z_1z_2z_3z_4\,z'_1z'_2z'_3z'_4$ in the alphabet~$\{x,y\}$,
let 
\[
g_i = 1+z_i u_i, \ \ h_i = 1+z'_iv_i. 
\]
In this setting, the Multiplicative Rule of Three (with $N=4$)
holds for 222 of the $2^8=256$ choices of~$\mathbf{z}$,
and fails for the remaining~34.
Specifically, the rule holds unless
\begin{itemize}
\item
$z_1=z_1'$, \ $z_2=z_2'$, \ $z_3=z_3'$, \ $z_4=z_4'$, \ or
\item
$z_1<z_1'$, \ $z_2\ge z_2'$, \ $z_3\ge z_3'$, \ $z_4<z_4'$, \ or
\item
$z_1>z_1'$, \ $z_2\le z_2'$, \ $z_3\le z_3'$, \ $z_4>z_4'$,
\end{itemize}
where we use the order relation $x<y$ on the symbols $x$ and~$y$.
This was shown using a noncommutative Gr\"obner basis calculation in \texttt{Magma}~\cite{Magma}.
\end{example}


The Multiplicative Rule of Three always
holds in the simplified version of Example~\ref{ex:256patterns} wherein the
factors  $g_i, h_i$ depend on a single set of noncommuting variables:


\begin{theorem}
\label{th:mrot-linear-uu}
The Multiplicative Rule of Three
holds
when $g_i = 1+(\alpha_i x + \alpha_i'y) u_i$ and $h_i = 1+(\beta_{i}x+\beta_i' y)u_i$
for any  $\alpha_i, \alpha_i', \beta_i, \beta_i' \in \QQ$.
(Here we use the notational conventions of Definition~\ref{def:mult-rot}.)
\end{theorem}

Theorem~\ref{th:mrot-linear-uu} is a special case of a more general result,
see Theorem~\ref{t pairs commute v0 2}.

Since the Multiplicative Rule of Three does not always hold, 
it is natural to consider Rules of Four and beyond 
(though this has not been the main focus of our investigation).  
For example, we can generalize Definition \ref{def:mult-rot} as follows: 
for any  $k \ge 0$, we say that the \emph{Multiplicative Rule of k holds} for 
$g_1,\ldots,g_N,h_1,\ldots,h_N$ if for any quotient ring
$R=\mathcal{A}/I$, the following are equivalent:
\begin{itemize}
\item $g_S h_S \equiv h_S g_S \bmod I$\ \ for all subsets $S\subset\{1,\dots,N\}$;
\item $g_S h_S \equiv h_S g_S \bmod I$\ \ for all subsets $S$ of cardinality $\le k$.
\end{itemize}
\begin{conjecture}
\label{cj ro4}
For any  $k \ge 0$ and $N > k$, the Multiplicative Rule of $k$ fails for  $g_i = 1+xu_i$,  $h_i = 1+x v_i$,  $i = 1, 2, \dots, N$.  
\end{conjecture}

Conjecture~\ref{cj ro4} would imply the failure of
the ``group-theoretic Rule of~$k$,'' for any~$k$: 

\begin{conjecture}
\label{cj ro4 group}
Fix integers~$k\ge0$ and $N > k$, and 
consider the group whose presentation is given by generators $g_1, \ldots, g_N,
h_1, \ldots, h_N$ and relations $g_Sh_S = h_Sg_S$ for  $|S| \le k$.
Then $g_Sh_S \neq h_Sg_S$ for any $|S| > k$. 
\end{conjecture}

We verified Conjecture~\ref{cj ro4} (hence Conjecture~\ref{cj ro4 group}) 
in the cases  $k \le 5$ via a noncommutative Gr\"obner basis calculation.

\section{Rules of Three for sums and products}
\label{sec:rot-invertible}

\begin{definition}
Let $M$ be a monoid.
A~subset $M'\!\subset\! M$ is called \emph{potentially invertible}
if $M$ can be embedded into a larger monoid in which all elements of~$M'$
have left and right inverses (necessarily equal to each other);
see \cite[Section~VII.3]{cohn-universal}.
Similarly, a subset~$R'$ of a ring~$R$ is potentially invertible
if $R$ can be embedded into a larger ring in which all elements of $R'$ are invertible.

By abuse of terminology, we say
that elements $g_1,\dots,g_N$
of a monoid (or ring) are potentially
invertible if $\{g_1,\dots,g_N\}$ is a potentially invertible subset.
\end{definition}

Let $R[x]$ be the ring of polynomials in one formal (central) variable~$x$,
with coefficients in a ring~$R$.
Then any subset of~$R[x]$ consisting of polynomials with constant term~$1$ is potentially invertible.
Indeed, $R[x]$~can be embedded into the ring $R[[x]]$ of formal power series over~$R$,
and each polynomial with constant term~$1$ is invertible in~$R[[x]]$.


Throughout this paper, we use the notation $[g,h]=gh-hg$
for the commutator of elements $g,h$
of an associative ring.

\begin{theorem}[The Rule of Three for sums \emph{vs.}\ products]
\label{t half ROT}
Let $R$ be a ring, and let
\hbox{$v_1, \ldots, v_N \!\in\! R$} and $g_1, \ldots, g_N\!\in \!R$,
with $g_1,\dots,g_N$ potentially invertible.
Then the following are equivalent:
\begin{itemize}
\item
$\big[\sum_{i \in S}v_i, g_S\big] = 0$ for all subsets
  $S\subset\{1,\dots,N\}$;
\item
$\big[\sum_{i \in S}v_i, g_S\big] = 0$ for all subsets
$S$ of cardinality~$\le 3$.
\end{itemize}
\end{theorem}

Theorem~\ref{t half ROT} can be generalized
to a setting of algebras with derivations.
Recall~that a \emph{derivation} on a $\QQ$-algebra~$R$
is a $\QQ$-linear map $\partial:R\to R$
satisfying Leibniz's law
\[
\partial(fg)=\partial(f)g+f\partial(g).
\]

\begin{theorem}[The Rule of Three for derivations]
\label{th:derivations-ROT}
Let $R$ be a $\QQ$-algebra.
Let $\partial_1,\ldots,\partial_N$ be
derivations on~$R$,
and let $g_1, \ldots, g_N$ be potentially invertible elements of $R$ satisfying
\begin{align}
\partial_a(g_a) &= 0  \qquad \text{for all $1 \le a \le N$}; \label{der e halfrot 1}\\
(\partial_b + \partial_a)(g_bg_a) &= 0 \qquad \text{for all $1 \leq a < b \leq N$}; \label{der e halfrot 2}\\
(\partial_c + \partial_b + \partial_a)(g_cg_bg_a) &= 0 \qquad \text{for all $1 \leq a < b < c \leq N$}. \label{der e halfrot 3}
\end{align}
Then 
$(\partial_N + \partial_{N-1} + \cdots + \partial_1)(g_Ng_{N-1}\cdots
g_1) = 0$.
\end{theorem}

\begin{theorem}[The Rule of Three for products \emph{vs.}\ products
    and sums]
\label{t half group ROT}
Let $R$ be a~ring, \linebreak[3]
and let $g_1,\ldots, g_N, h_1,\ldots, h_N\in R$ be potentially invertible. 
Then the following are equivalent:
 \begin{itemize}
\item $[g_S, h_S]=[g_S, \sum_{i \in S}h_i]=0$ for all subsets $S\subset\{1,\dots,N\}$;
\item $[g_S, h_S]=[g_S, \sum_{i \in S}h_i]=0$ for all subsets $S$ of cardinality $\le 3$.
\end{itemize}
\end{theorem}

\begin{theorem}
\label{t 2variable linear strengthened}
Let $R$ be a ring,
and let $g_1, \ldots, g_N, h_1, \ldots, h_N\in R$
be potentially invertible elements satisfying the relations
\begin{align}
[h_a,g_a] &= 0  \qquad \text{for all $1 \le a \le N$}; \label{e linear 1}\\
[h_b+h_a,g_b+g_a] &= 0  \qquad \text{for all $1 \le a < b  \le N$}; \label{e linear 2sum}\\
[h_b + h_a, g_bg_a] &= 0 \qquad \text{for all $1 \leq a < b \leq
  N$}; \label{e linear 22}
\\
[h_c + h_b + h_a, g_cg_bg_a] &= 0 \qquad \text{for all $1 \leq a < b < c
  \leq N$};
\label{e linear 33}
\\
[g_b + g_a, h_bh_a] &= 0 \qquad \text{for all $1 \leq a < b \leq N$}; \label{e linear 2b}\\
[g_c + g_b + g_a, h_ch_bh_a] &= 0 \qquad \text{for all $1 \leq a < b < c \leq N$}. \label{e linear 3b}
\end{align}
Then   $\left[\textstyle\sum_{i \in S} g_i, \textstyle\sum_{i \in S} h_i \right]
\!=\! \left[\textstyle\sum_{i \in S} g_i, h_S \right]
\!=\! \left[g_S, \textstyle\sum_{i \in S} h_i \right]
\!=\![g_S, h_S]
\!=\! 0 $ \, for all subsets $S$.  
\end{theorem}

\begin{remark}
\label{rem:add-rule-of-two}
It is easy to  see that the relations \eqref{e linear 1}--\eqref{e linear 2sum}
alone imply
\[
\left[\textstyle\sum_{i \in S} g_i, \textstyle\sum_{i \in S} h_i \right] = 0
\]
for all $S \subset \{1,\ldots, N\}$;
this implication can be regarded as an ``Additive Rule of Two.''
\end{remark}

Theorem~\ref{t 2variable linear strengthened} implies (and so
can be regarded as a strengthening of) the following rule.


\begin{corollary}[The Rule of Three for  products and sums
\emph{vs.}\ products and sums]
\label{cor:*+vs*+}
Let $R$ be a ring, and let $g_1,\ldots, g_N, h_1,\ldots, h_N\in R$ be  potentially
invertible.
Then the following are equivalent:
\begin{itemize}
\item
$\left[\sum_{i \in S} g_i, \sum_{i \in S} h_i \right] =
\left[\sum_{i \in S} g_i, h_S \right] = \left[g_S, \sum_{i \in S}
    h_i \right] = [g_S, h_S] = 0$ for all $S$;
\item
$\left[\sum_{i \in S} g_i, \sum_{i \in S} h_i \right] =
  \left[\sum_{i \in S} g_i, h_S \right] = \left[g_S, \sum_{i \in S}
    h_i \right] = [g_S, h_S] = 0$ for all $|S| \le 3$.
\end{itemize}
\end{corollary}

We note that Corollary~\ref{cor:*+vs*+} is also immediate from
Theorem~\ref{t half group ROT} and Remark~\ref{rem:add-rule-of-two}.



\begin{remark}
\label{rem invertibility needed}
The cases $|S|\le 3$ of
Theorem~\ref{t 2variable linear
    strengthened} hold without the requirement of potential
invertibility; this can be verified by a noncommutative Gr\"obner
basis calculation.
However, for \mbox{$|S|\ge 4$}, this requirement cannot be dropped.
More precisely, in the free associative algebra $\QQ\langle
g_1, g_2, g_3, g_4, h_1, h_2, h_3, h_4 \rangle$,
the two-sided ideal generated by the left-hand sides of
\eqref{e linear 1}--\eqref{e linear 3b} does not contain the element
$[g_4g_3g_2g_1,h_4h_3h_2h_1]$
(nor does it contain \mbox{$[h_4+h_3+h_2+h_1, g_4g_3g_2g_1]$}).
This was checked using a noncommutative Gr\"obner basis calculation in
\texttt{Magma}~\cite{Magma}.
\end{remark}

\begin{remark}
\label{rem:e1-comm-with-e2}
As explained in Section~\ref{sec:main-proofs},
Theorem~\ref{t 2variable linear strengthened}
directly implies Theorem~\ref{t es commute uv 2} via the substitutions
$g_i =1+xu_i$,  $h_i = 1+yv_i$.
On the other hand, if we think of $g_i $ and $h_i$ in Theorem~\ref{t
  2variable linear strengthened} as $u_i$ and $v_i$,
then Theorems~\ref{t 2variable linear strengthened} and~\ref{t es
  commute uv 2}
are ``incomparable'':
in Theorem~\ref{t 2variable linear strengthened},
we do not need the relations
\begin{equation}
\label{eq:e1-comm-with-e2}
[h_c + h_b + h_a, g_cg_b + g_cg_a + g_bg_a] = 0,
\end{equation}
but we do require $g_i$ and $h_i$ to be potentially invertible
(cf.\ Remark~\ref{rem invertibility needed}),
while Theorem~\ref{t es commute uv 2} requires  $[e_1(\mathbf{v}_S),
e_2(\mathbf{u}_S)] = 0$ for  $|S| = 3$ but not invertibility.

To further clarify matters,
we note that the assumptions of Theorem~\ref{t 2variable linear strengthened}
(including invertibility) do not imply~\eqref{eq:e1-comm-with-e2}.
To see this, take $N\!=\!3$,
$g_i = 1+x^3u_i^3 + x^4u_i^4 + x^5u_i^5$,
and $h_i = 1+yv_i$ for $i = 1, 2, 3$.
Impose relations on the $u_i, v_i$ derived from the conditions $g_Sh_S = h_Sg_S$ for all
$S$ of cardinality~$\le 3$.
Then relations \eqref{e linear 1}--\eqref{e linear 3b} hold
but \eqref{eq:e1-comm-with-e2} (with $(a,b,c)=(1,2,3)$) does not.
This was checked via a noncommutative Gr\"obner basis computation.
\end{remark}

Theorem \ref{t half ROT} and
Theorems \ref{t half quadratic half group ROT}--\ref{t half group ROTv2} below
form a natural progression.

\begin{theorem}[The Rule of Three for products \emph{vs.}\ sums and quadratic forms]
\label{t half quadratic half group ROT}
Let $R$ be a ring.
Let $v_1, \ldots, v_N \in R$ and $g_1, \ldots, g_N\in R$,
with $g_1,\dots,g_N$ potentially invertible.
Then the following are equivalent:
\begin{itemize}
\item
$g_S\, e_\ell(\mathbf{v}_S) = e_\ell(\mathbf{v}_S)\, g_S$ for all subsets
  $S\subset\{1,\dots,N\}$ and $\ell\le 2$;
\item
$g_S\, e_\ell(\mathbf{v}_S) = e_\ell(\mathbf{v}_S)\, g_S$ for all subsets
$S$ of cardinality~$\le 3$ and $\ell\le 2$.
\end{itemize}
\end{theorem}

\begin{theorem}[The Rule of Three for products \emph{vs.}\ elementary symmetric functions]
\label{t half group ROTv2}
Let~$R$ be a ring.
Let $v_1, \ldots, v_N \in R$ and $g_1, \ldots, g_N\in R$,
with $g_1,\dots,g_N$ potentially invertible.
Then the following are equivalent:
\begin{itemize}
\item
$g_S\, e_\ell(\mathbf{v}_S) = e_\ell(\mathbf{v}_S)\, g_S$ for all subsets
  $S\subset\{1,\dots,N\}$ and all~$\ell$;
\item
$g_S\, e_\ell(\mathbf{v}_S) = e_\ell(\mathbf{v}_S)\, g_S$ for all subsets
$S$ of cardinality~$\le 3$ and all~$\ell$.
\end{itemize}
\end{theorem}

\begin{remark}
Theorem~\ref{t half group ROTv2} implies a variant of the
Multiplicative Rule of Three
(see Corollary~\ref{c half group ROT}) which generalizes
Theorem~\ref{t half group ROT corollary mini}
(which in turn generalizes Corollary~\ref{cor:rule-of-3-uv-linear}).
%
\end{remark}




\section{Dehn diagrams. 
Group-theoretic lemmas}
\label{sec:rot-groups}

For the purposes of this paper, a Dehn diagram
(a simplified version of the notion of van Kampen diagram,
see, e.g., \cite[Section~4]{peifer})
is a planar oriented graph whose edges are labeled by elements of a
group, so that each cycle corresponds to a relation in the group.
A more precise formulation is given in
Definition~\ref{def:dehn-diagram} below.

\begin{definition}
\label{def:dehn-diagram}
Let $D$ be a finite oriented graph properly embedded in the real plane;
that is, it is drawn so that its edges only meet at common endpoints.
We require each vertex of $D$ to have at least two incident edges.
The complement of $D$ in the plane is a dis\-joint union of \emph{faces}:
some \emph{bounded faces} homeomorphic to disks, and a single \emph{outer~face}.

Assume that  every edge of~$D$ has been labeled
by an element
of a group~$G$.
Such an edge-labeled oriented graph is called a \emph{Dehn diagram}
if the product along the boundary~$\partial F$ of each bounded face~$F$ is equal to~$1$.
More precisely, starting with an arbitrary vertex on~$\partial F$
and moving either clockwise or counterclockwise, we multiply the
elements of~$G$ associated with the edges, inverting them when
moving against the orientation of an~edge.
It is easy to see that this condition does not depend on the starting
location on~$\partial F$.
\end{definition}

The following simple but useful observation goes back to M.~Dehn.

\begin{lemma}
\label{lem:dehn-outer}
In a Dehn diagram, the product of labels along the boundary of
the outer face is equal to~$1$.
\end{lemma}

Below we present several group-theoretic results in the spirit of
(multiplicative) Rules of
Three, cf.\ Problem~\ref{problem:group theoretic ROT}.
All the proofs utilize Dehn diagrams.

\begin{proposition}
\label{pr:malcev-cba}
Let $G$ be a group, and let $a,b,c,A,B,C\in G$ satisfy
\begin{equation}
\label{eq:cbaCBA}
aC=Ca, \quad
bB=Bb,\quad
cA=Ac, \quad
baBA=BAba, \quad
cbCB=CBcb.
\end{equation}
Then
$\,cba\,CBA=CBA\,cba$.
\end{proposition}

\begin{proof}
In the Dehn diagram shown in Figure~\ref{fig:malcev-cba},
each bounded face commutes.
Hence so does the outer face, and the claim follows.
\end{proof}

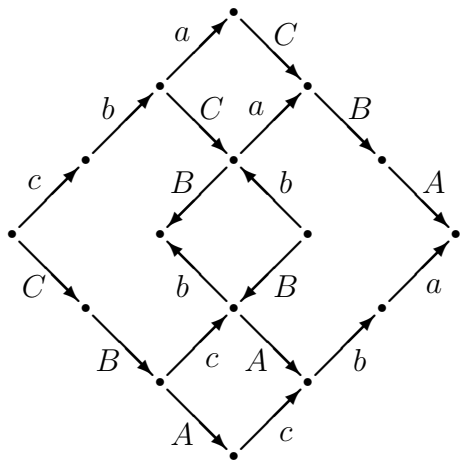
\begin{figure}[ht]
\begin{center}
\setlength{\unitlength}{2.8pt}
\begin{picture}(60,57)(0,-30)

\multiput(0,0)(20,0){4}{\circle*{1}}
\multiput(10,10)(20,0){3}{\circle*{1}}
\multiput(10,-10)(20,0){3}{\circle*{1}}
\multiput(20,20)(20,0){2}{\circle*{1}}
\multiput(20,-20)(20,0){2}{\circle*{1}}
\put( 30,30){\circle*{1}}
\put( 30,-30){\circle*{1}}

\thicklines
\multiput(1,1)(10,10){3}{\vector(1,1){8}}
\multiput(31,-29)(10,10){3}{\vector(1,1){8}}
\multiput(1,-1)(10,-10){3}{\vector(1,-1){8}}
\multiput(31,29)(10,-10){3}{\vector(1,-1){8}}
\multiput(29,-9)(10,10){2}{\vector(-1,1){8}}
\multiput(29,9)(10,-10){2}{\vector(-1,-1){8}}
\put(21,19){\vector(1,-1){8}}
\put(31,-11){\vector(1,-1){8}}
\put(31,11){\vector(1,1){8}}
\put(21,-19){\vector(1,1){8}}

\put(3,7){\makebox(0,0){$c$}}
\put(13,17){\makebox(0,0){$b$}}
\put(23,27){\makebox(0,0){$a$}}

\put(37,27){\makebox(0,0){$C$}}
\put(47,17){\makebox(0,0){$B$}}
\put(57,7){\makebox(0,0){$A$}}

\put(23,7){\makebox(0,0){$B$}}
\put(37,-7){\makebox(0,0){$B$}}
\put(23,-7){\makebox(0,0){$b$}}
\put(37,7){\makebox(0,0){$b$}}

\put(3,-7){\makebox(0,0){$C$}}
\put(13,-17){\makebox(0,0){$B$}}
\put(23,-27){\makebox(0,0){$A$}}

\put(27,17){\makebox(0,0){$C$}}
\put(27,-17){\makebox(0,0){$c$}}

\put(33,17){\makebox(0,0){$a$}}
\put(33,-17){\makebox(0,0){$A$}}

\put(37,-27){\makebox(0,0){$c$}}
\put(47,-17){\makebox(0,0){$b$}}
\put(57,-7){\makebox(0,0){$a$}}

\end{picture}
\end{center}
\caption{\label{fig:malcev-cba}
The proof of Proposition~\ref{pr:malcev-cba}.
}
\end{figure}

Proposition~\ref{pr:malcev-cba} is a special case of the following result.

\begin{theorem}
\label{t pairs commute simple}
Let $G$ be a group, and let the elements
$g_1,\dots,g_N, h_1,\dots, h_N\in G$ satisfy
\begin{alignat}{2}
\label{eq:distant-commute}
g_a h_c &= h_c g_a && \quad \text{for all $a,c\in\{1,\dots, N\}$ with
  $|a-c| \ge 2$};\\
g_a h_a &= h_a g_a  & &\quad
\text{for all $1< a< N$;} \label{eq:gihi=higi}
\\
g_{a+1} g_a h_{a+1} h_a &= h_{a+1} h_a g_{a+1} g_a  & &\quad
\text{for all $1\le a< N$.} \label{eq:two-consec-commute}
\end{alignat}
Then
\begin{equation}
g_N g_{N-1} \cdots g_1\, h_N h_{N-1}\cdots h_1
= h_N h_{N-1} \cdots h_1 \,g_N g_{N-1} \cdots g_1.
\label{eq:gN...g1-commutes-with-hN...h1}
\end{equation}
\end{theorem}

\begin{proof}
The proof is a direct generalization of the above proof of
Proposition~\ref{pr:malcev-cba}.
It relies on the Dehn diagram shown in Figure~\ref{f pairs commute
  simple}, which illustrates the case $N=6$.
The quadrilateral and octagonal faces of the diagram correspond to
relations~\eqref{eq:gihi=higi}
and~\eqref{eq:two-consec-commute}, respectively.
The bounded faces at the bottom and the top can be tiled by
rhombi corresponding to the
relations~\eqref{eq:distant-commute};
to avoid clutter, these tiles are not shown.
The outer boundary corresponds to~\eqref{eq:gN...g1-commutes-with-hN...h1}.
\end{proof}

\begin{figure}[ht]
\begin{tikzpicture}[xscale = 1,yscale = 0.84]
\tikzstyle{vertex}=[circle, fill, inner sep=1pt, outer sep=2pt]
\tikzstyle{framedvertex}=[inner sep=3pt, outer sep=4pt, draw=gray]
\tikzstyle{circledvertex}=[ellipse, draw, inner sep=1pt, outer sep=3pt, draw=gray]
\tikzstyle{aedge} = [draw, thick, -latex,black, inner sep = 2pt]
\tikzstyle{edge} = [draw, thick, -,black]
\tikzstyle{doubleedge} = [draw, thick, double distance=1pt, -,black]
\tikzstyle{hiddenedge} = [draw=none, thick, double distance=1pt, -,black]
\tikzstyle{dashededge} = [draw, very thick, dashed, black]
\tikzstyle{LabelStyleH} = [text=black, anchor=south]
\tikzstyle{LabelStyleHn} = [text=black, anchor=north]
\tikzstyle{LabelStyleV} = [text=black, anchor=east]
\tikzstyle{LabelStyleNE} = [text=black, inner sep = 0.5pt, anchor=north east]
\tikzstyle{LabelStyleNEn} = [text=black, inner sep = 0.5pt,  anchor=south west]
\tikzstyle{LabelStyleNW} = [text=black, inner sep = 0.5pt, anchor=north west]
\tikzstyle{LabelStyleNWn} = [text=black, inner sep = 0.5pt, anchor=south east]

\node[vertex] (v73) at (7,3){};
\node[vertex] (v62) at (6,2){};
\node[vertex] (v51) at (5,1){};
\node[vertex] (v40) at (4,0){};
\node[vertex] (v3a) at (3,-1){};

\node[vertex] (v2b) at (2,-2){};

\node[vertex] (v1a) at (1,-1){};
\node[vertex] (v00) at (0,0){};
\node[vertex] (va1) at (-1,1){};
\node[vertex] (vb2) at (-2,2){};
\node[vertex] (vc3) at (-3,3){};

\node[vertex] (vd4) at (-4,4){};

\node[vertex] (vc5) at (-3,5){};
\node[vertex] (vb6) at (-2,6){};
\node[vertex] (va7) at (-1,7){};

\node[vertex] (v22) at (2,2){};
\node[vertex] (v33) at (3,3){};
\node[vertex] (v35) at (3,5){};
\node[vertex] (v26) at (2,6){};
\node[vertex] (v08) at (0,8){};
\node[vertex] (v19) at (1,9){};
\node[vertex] (v210) at (2,10){};

\node[vertex] (v39) at (3,9){};
\node[vertex] (v48) at (4,8){};
\node[vertex] (v57) at (5,7){};
\node[vertex] (v66) at (6,6){};
\node[vertex] (v75) at (7,5){};
\node[vertex] (v84) at (8,4){};

\node[vertex] (va5) at (-1,5){};
\node[vertex] (v06) at (0,6){};
\node[vertex] (v15) at (1,5){};
\node[vertex] (v26) at (2,6){};
\node[vertex] (v35) at (3,5){};
\node[vertex] (v46) at (4,6){};
\node[vertex] (v55) at (5,5){};
\node[vertex] (v66) at (6,6){};

\node[vertex] (vb2) at (-2,2){};
\node[vertex] (va3) at (-1,3){};
\node[vertex] (v02) at (0,2){};
\node[vertex] (v13) at (1,3){};
\node[vertex] (v22) at (2,2){};
\node[vertex] (v33) at (3,3){};
\node[vertex] (v42) at (4,2){};
\node[vertex] (v53) at (5,3){};
\node[vertex] (v62) at (6,2){};

\node[vertex] (vb4) at (-2,4){};
\node[vertex] (v04a) at (-0.2,4){};
\node[vertex] (v04b) at (+0.2,4){};
\node[vertex] (v24a) at (1.8,4){};
\node[vertex] (v24b) at (2.2,4){};
\node[vertex] (v44a) at (3.8,4){};
\node[vertex] (v44b) at (4.2,4){};
\node[vertex] (v64) at (6,4){};


\draw[aedge] (vd4) to node[LabelStyleNE]{\footnotesize $h_6$} (vc3);
\draw[aedge] (vc3) to node[LabelStyleNE]{\footnotesize $h_5$} (vb2);
\draw[aedge] (vb2) to node[LabelStyleNE]{\footnotesize $h_4$} (va1);
\draw[aedge] (va1) to node[LabelStyleNE]{\footnotesize $h_3$} (v00);
\draw[aedge] (v00) to node[LabelStyleNE]{\footnotesize $h_2$} (v1a);
\draw[aedge] (v1a) to node[LabelStyleNE]{\footnotesize $h_1$} (v2b);

\draw[aedge] (v2b) to node[LabelStyleNW]{\footnotesize $g_6$} (v3a);
\draw[aedge] (v3a) to node[LabelStyleNW]{\footnotesize $g_5$} (v40);
\draw[aedge] (v40) to node[LabelStyleNW]{\footnotesize $g_4$} (v51);
\draw[aedge] (v51) to node[LabelStyleNW]{\footnotesize $g_3$} (v62);
\draw[aedge] (v62) to node[LabelStyleNW]{\footnotesize $g_2$} (v73);
\draw[aedge] (v73) to node[LabelStyleNW]{\footnotesize $g_1$} (v84);

\draw[aedge] (vd4) to node[LabelStyleNWn]{\footnotesize $g_6$} (vc5);
\draw[aedge] (vc5) to node[LabelStyleNWn]{\footnotesize $g_5$} (vb6);
\draw[aedge] (vb6) to node[LabelStyleNWn]{\footnotesize $g_4$} (va7);
\draw[aedge] (va7) to node[LabelStyleNWn]{\footnotesize $g_3$} (v08);
\draw[aedge] (v08) to node[LabelStyleNWn]{\footnotesize $g_2$} (v19);
\draw[aedge] (v19) to node[LabelStyleNWn]{\footnotesize $g_1$} (v210);

\draw[aedge] (v210) to node[LabelStyleNEn]{\footnotesize $h_6$} (v39);
\draw[aedge] (v39) to node[LabelStyleNEn]{\footnotesize $h_5$} (v48);
\draw[aedge] (v48) to node[LabelStyleNEn]{\footnotesize $h_4$} (v57);
\draw[aedge] (v57) to node[LabelStyleNEn]{\footnotesize $h_3$} (v66);
\draw[aedge] (v66) to node[LabelStyleNEn]{\footnotesize $h_2$} (v75);
\draw[aedge] (v75) to node[LabelStyleNEn]{\footnotesize $h_1$} (v84);


\draw[aedge] (vb2) to node[LabelStyleNW]{\footnotesize \footnotesize $g_6$} (va3);
\draw[aedge] (va3) to node[LabelStyleNE]{\footnotesize $h_4$} (v02);
\draw[aedge] (v02) to node[LabelStyleNW]{\footnotesize $g_5$} (v13);
\draw[aedge] (v13) to node[LabelStyleNE]{\footnotesize $h_3$} (v22);
\draw[aedge] (v22) to node[LabelStyleNW]{\footnotesize $g_4$} (v33);
\draw[aedge] (v33) to node[LabelStyleNE]{\footnotesize $h_2$} (v42);
\draw[aedge] (v42) to node[LabelStyleNW]{\footnotesize $g_3$} (v53);
\draw[aedge] (v53) to node[LabelStyleNE]{\footnotesize $h_1$} (v62);

\draw[aedge] (va3) to node[LabelStyleNE]{\footnotesize $g_5$} (vb4);
\draw[aedge] (v04a) to node[LabelStyleNW]{\footnotesize $h_5$} (va3);
\draw[aedge] (v13) to node[LabelStyleNE]{\footnotesize $g_4$} (v04b);
\draw[aedge] (v24a) to node[LabelStyleNW]{\footnotesize $h_4$} (v13);
\draw[aedge] (v33) to node[LabelStyleNE]{\footnotesize $g_3$} (v24b);
\draw[aedge] (v44a) to node[LabelStyleNW]{\footnotesize $h_3$} (v33);
\draw[aedge] (v53) to node[LabelStyleNE]{\footnotesize $g_2$} (v44b);
\draw[aedge] (v64) to node[LabelStyleNW]{\footnotesize $h_2$} (v53);

\draw[aedge] (va5) to node[LabelStyleNWn]{\footnotesize $h_5$} (vb4);
\draw[aedge] (v04a) to node[LabelStyleNEn]{\footnotesize $g_5$} (va5);
\draw[aedge] (v15) to node[LabelStyleNWn]{\footnotesize $h_4$} (v04b);
\draw[aedge] (v24a) to node[LabelStyleNEn]{\footnotesize $g_4$} (v15);
\draw[aedge] (v35) to node[LabelStyleNWn]{\footnotesize $h_3$} (v24b);
\draw[aedge] (v44a) to node[LabelStyleNEn]{\footnotesize $g_3$} (v35);
\draw[aedge] (v55) to node[LabelStyleNWn]{\footnotesize $h_2$} (v44b);
\draw[aedge] (v64) to node[LabelStyleNEn]{\footnotesize $g_2$} (v55);

\draw[aedge] (vb6) to node[LabelStyleNEn]{\footnotesize $h_6$} (va5);
\draw[aedge] (va5) to node[LabelStyleNWn]{\footnotesize $g_4$} (v06);
\draw[aedge] (v06) to node[LabelStyleNEn]{\footnotesize $h_5$} (v15);
\draw[aedge] (v15) to node[LabelStyleNWn]{\footnotesize $g_3$} (v26);
\draw[aedge] (v26) to node[LabelStyleNEn]{\footnotesize $h_4$} (v35);
\draw[aedge] (v35) to node[LabelStyleNWn]{\footnotesize $g_2$} (v46);
\draw[aedge] (v46) to node[LabelStyleNEn]{\footnotesize $h_3$} (v55);
\draw[aedge] (v55) to node[LabelStyleNWn]{\footnotesize $g_1$} (v66);
\end{tikzpicture}
\caption{\label{f pairs commute simple}
Proof of Theorem~\ref{t pairs commute simple}.}
\end{figure}
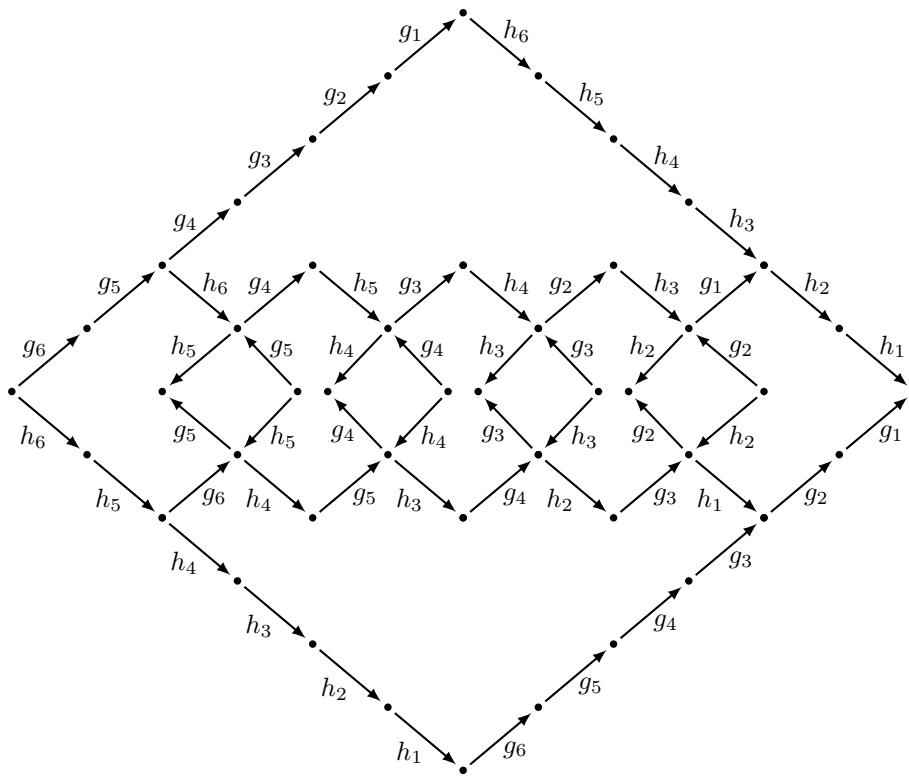

\begin{remark}
The statement of Proposition~\ref{pr:malcev-cba} (or its
generalization, Theorem~\ref{t pairs commute simple})
does not involve inverses.
As such, it readily extends to potentially invertible elements 
$a,b,c,A,B,C$ in a monoid~$M$.
Note however that the assumption of potential invertibility cannot be
dropped: in an arbitrary monoid~$M$,
elements $a,b,c,A,B,C$ satisfying relations~\eqref{eq:cbaCBA}
do not have to satisfy $cbaCBA=CBAcba$.
On the other hand, the latter identity is implied by~\eqref{eq:cbaCBA} under the additional
assumption that~$B$ is potentially invertible.
(There is no need to require anything else.)

%
%
\end{remark}

\begin{remark}
A.~I.~Malcev~\cite{malcev-1, malcev-2} gave an infinite list of
conditions (more precisely, \emph{quasi-identities})
that a monoid~$M$ must satisfy in order to be embeddable into a group.
(Note that embeddability of $M$ into a group is equivalent to
the set of all elements of $M$ being potentially invertible.
Malcev's argument is reproduced in
\cite[Section~VII.3]{cohn-universal};
see also~\cite{lambek} for an alternative perspective.)
Apart from left and right cancellativity,
the simplest of those conditions is the following (cf.\ Figure~\ref{fig:malcev-van-kampen}):
\begin{equation}
\label{eq:malcev}
\forall \, p, q, r, s, P,Q,R,S\in M\quad
((PQ=pq{\ \&\ }
RQ=rq{\ \&\ }
RS=rs) \Rightarrow
PS=ps).
\end{equation}
%
%
It turns out that Theorem~\ref{t pairs commute simple} holds for any
monoid satisfying condition~\eqref{eq:malcev}.
Curiously, neither cancellativity nor other Malcev's conditions are required.

\begin{figure}[ht]
\begin{center}
\setlength{\unitlength}{2.2pt}
\begin{picture}(60,38)(0,-20)

\put( 0,0){\circle*{2}}
\put( 20,0){\circle*{2}}
\put( 40,0){\circle*{2}}
\put( 60,0){\circle*{2}}
\put( 30,20){\circle*{2}}
\put( 30,-20){\circle*{2}}

\thicklines
\put(39,2){\vector(-1,2){8}}
\put(39,-2){\vector(-1,-2){8}}
\put(29,18){\vector(-1,-2){8}}
\put(29,-18){\vector(-1,2){8}}
\put(33,18){\vector(3,-2){24.5}}
\put(33,-18){\vector(3,2){24.5}}
\put(2.3,-1.4){\vector(3,-2){25.6}}
\put(2.3,1.4){\vector(3,2){25.6}}

\put(15,14){\makebox(0,0){$P$}}
\put(15,-14){\makebox(0,0){$p$}}
\put(45,14){\makebox(0,0){$S$}}
\put(45,-14){\makebox(0,0){$s$}}

\put(21,9){\makebox(0,0){$Q$}}
\put(21,-9){\makebox(0,0){$q$}}
\put(39,9){\makebox(0,0){$R$}}
\put(39,-9){\makebox(0,0){$r$}}

\end{picture}
\end{center}
\caption{\label{fig:malcev-van-kampen}
Dehn diagram illustrating Malcev's condition~\eqref{eq:malcev}.
}
\end{figure}
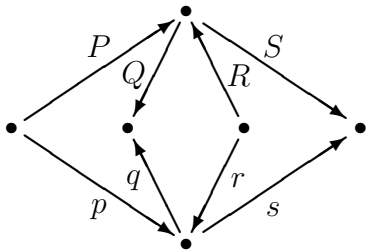
\end{remark}

\begin{theorem}
\label{thm:group-theoretic-gi*hi}
Let $G$ be a group, and let
$g_1, \ldots, g_N, h_1, \ldots, h_N\in G$ be such that
for $1<b<N$, we have:
\begin{equation}
\label{eq:z-commutes-with-gihi}
\text{if $z\in G$ commutes with
$h_b^{-1} g_b$, then $z$ commutes with both $g_b$ and~$h_b$.}
\end{equation}
Then the following are equivalent:
 \begin{itemize}
\item $g_S h_S = h_S g_S$ for all subsets $S\subset\{1,\dots,N\}$;
\item $g_S h_S = h_S g_S$ for all subsets $S$ of cardinality $\le 3$.
\end{itemize}
\end{theorem}

In Section~\ref{sec:proof:u=v}, we show that condition~\eqref{eq:z-commutes-with-gihi} is
satisfied in the setting of Theorem~\ref{t pairs commute v0}.
This enables us to deduce the latter from
Theorem~\ref{thm:group-theoretic-gi*hi}.

The proof of Theorem~\ref{thm:group-theoretic-gi*hi}
relies on two group-theoretic lemmas.

\begin{lemma}
\label{l pairs commute group only}
Let $G$ be a group, and let
$g_a, g_b, g_c, h_a, h_b, h_c\in G$ satisfy
\begin{alignat}{3}
g_ah_a &= h_ag_a,\quad &
g_bh_b &= h_bg_b,\quad &
g_ch_c &= h_cg_c, \label{e ab 1vars 3 only}\\
g_{b}g_ah_{b}h_a &= h_{b}h_ag_{b}g_a,\quad&
g_{c}g_ah_{c}h_a &= h_{c}h_ag_{c}g_a,\quad&
g_{c}g_bh_{c}h_b &= h_{c}h_bg_{c}g_b. \label{e ab 2vars 3 only}
\end{alignat}
Then, if one of the following two relations holds in  $G$, so does the other:
\begin{align}
h_b^{-1}g_b(g_c^{-1}h_ag_ch_a^{-1}) &= (g_c^{-1}h_ag_ch_a^{-1})h_b^{-1}g_b; \label{e ab 3vars trick 3 only} \\
g_c g_{b}g_a h_c h_{b}h_a &= h_c h_{b}h_a g_c g_{b}g_a. \label{e ab Nvars 3 only}
\end{align}
\end{lemma}
\begin{proof}
It suffices to observe that in the Dehn diagram in Figure~\ref{f 3vars},
\begin{itemize}
\item the outer face
corresponds to the relation~\eqref{e ab Nvars 3 only};
\item the $12$-gon in the middle corresponds to the relation~\eqref{e
  ab 3vars trick 3 only};
\item the other bounded faces correspond to relations
\eqref{e ab 1vars 3 only}--\eqref{e ab 2vars 3 only}. \qedhere
\end{itemize}
\end{proof}

\begin{figure}[ht]
\centerfloat
\begin{tikzpicture}[xscale = 2,yscale = 1.5]
\tikzstyle{vertex}=[circle, fill, inner sep=1pt, outer sep=2pt]
\tikzstyle{framedvertex}=[inner sep=3pt, outer sep=4pt, draw=gray]
\tikzstyle{circledvertex}=[ellipse, draw, inner sep=1pt, outer sep=3pt, draw=gray]
\tikzstyle{aedge} = [draw, thick, -latex,black, inner sep = 2pt]
\tikzstyle{edge} = [draw, thick, -,black]
\tikzstyle{doubleedge} = [draw, thick, double distance=1pt, -,black]
\tikzstyle{hiddenedge} = [draw=none, thick, double distance=1pt, -,black]
\tikzstyle{dashededge} = [draw, very thick, dashed, black]
\tikzstyle{LabelStyleH} = [text=black, anchor=south]
\tikzstyle{LabelStyleHn} = [text=black, anchor=north]
\tikzstyle{LabelStyleV} = [text=black, anchor=east]
\tikzstyle{LabelStyleVn} = [text=black, anchor=west]
\tikzstyle{LabelStyleNE} = [text=black, anchor=north east]
\tikzstyle{LabelStyleNEn} = [text=black, anchor=south west]
\tikzstyle{LabelStyleNW} = [text=black, anchor=north west]
\tikzstyle{LabelStyleNWn} = [text=black, anchor=south east]

\node[vertex] (vb0) at (-2,0){};
\node[vertex] (va0) at (-1,0){};
\node[vertex] (v00) at (0,0){};
\node[vertex] (v10) at (1,0){};
\node[vertex] (v20) at (2,0){};
\node[vertex] (vb1) at (-2,1){};
\node[vertex] (v21) at (2,1){};
\node[vertex] (vb2) at (-2,2){};
\node[vertex] (va2) at (-1,2){};
\node[vertex] (v02) at (0,2){};
\node[vertex] (v12) at (1,2){};
\node[vertex] (v22) at (2,2){};
\node[vertex] (vb3) at (-2,3){};
\node[vertex] (va3) at (-1,3){};
\node[vertex] (v13) at (1,3){};
\node[vertex] (v23) at (2,3){};

\node[vertex] (v04) at (0,4.5){};

\node[vertex] (v40) at (2,-1){};
\node[vertex] (v41) at (3.5,-1){};
\node[vertex] (v42) at (3.5,1.5){};
\node[vertex] (vd0) at (-2,-1){};
\node[vertex] (vd1) at (-3.5,-1){};
\node[vertex] (vd2) at (-3.5,1.5){};

\node[vertex] (vc1) at (-3,1){};

\draw[aedge] (v00) to node[LabelStyleH]{\footnotesize $g_c$} (v10);
\draw[aedge] (v20) to node[LabelStyleH]{\footnotesize $h_a$} (v10);
\draw[aedge] (va0) to node[LabelStyleH]{\footnotesize $g_c$} (vb0);
\draw[aedge] (va0) to node[LabelStyleH]{\footnotesize $h_a$} (v00);

\draw[aedge] (v02) to node[LabelStyleH]{\footnotesize $g_c$} (v12);
\draw[aedge] (v22) to node[LabelStyleH]{\footnotesize $h_a$} (v12);
\draw[aedge] (va2) to node[LabelStyleH]{\footnotesize $g_c$} (vb2);
\draw[aedge] (va2) to node[LabelStyleH]{\footnotesize $h_a$} (v02);

\draw[aedge] (v23) to node[LabelStyleH]{\footnotesize $h_a$} (v13);
\draw[aedge] (va3) to node[LabelStyleH]{\footnotesize $g_c$} (vb3);

\draw[aedge] (vb1) to node[LabelStyleV]{\footnotesize $h_b$} (vb0);
\draw[aedge] (v21) to node[LabelStyleV]{\footnotesize $h_b$} (v20);
\draw[aedge] (vb1) to node[LabelStyleV]{\footnotesize $g_b$} (vb2);
\draw[aedge] (v21) to node[LabelStyleV]{\footnotesize $g_b$} (v22);
\draw[aedge] (vb3) to node[LabelStyleV]{\footnotesize $h_c$} (vb2);
\draw[aedge] (va3) to node[LabelStyleV]{\footnotesize $h_c$} (va2);
\draw[aedge] (v12) to node[LabelStyleV]{\footnotesize $g_a$} (v13);
\draw[aedge] (v22) to node[LabelStyleV]{\footnotesize $g_a$} (v23);

\draw[aedge] (vb3) to node[LabelStyleNWn]{\footnotesize $g_a$} (v04);
\draw[aedge] (v04) to node[LabelStyleNEn]{\footnotesize $h_c$} (v23);

\draw[aedge] (v23) to node[LabelStyleNEn]{\footnotesize $h_b$} (v42);
\draw[aedge] (v42) to node[LabelStyleVn]{\footnotesize $h_a$} (v41);
\draw[aedge] (v40) to node[LabelStyleHn]{\footnotesize $g_a$} (v41);
\draw[aedge, bend right=0] (v10) to node[LabelStyleNE]{\footnotesize $g_b$} (v40);

\draw[aedge] (vd2) to node[LabelStyleNWn]{\footnotesize $g_b$} (vb3);
\draw[aedge] (vd1) to node[LabelStyleV]{\footnotesize $g_c$} (vd2);
\draw[aedge] (vd1) to node[LabelStyleHn]{\footnotesize $h_c$} (vd0);
\draw[aedge, bend right=0] (vd0) to node[LabelStyleNW]{\footnotesize $h_b$} (va0);

\draw[aedge] (vb2) to node[LabelStyleNWn]{\footnotesize $h_b$} (vc1);
\draw[aedge] (vb0) to node[LabelStyleNE]{\footnotesize $g_b$} (vc1);

\end{tikzpicture}
\caption{\label{f 3vars}The proof of Lemma \ref{l pairs commute group only}.}
\end{figure}
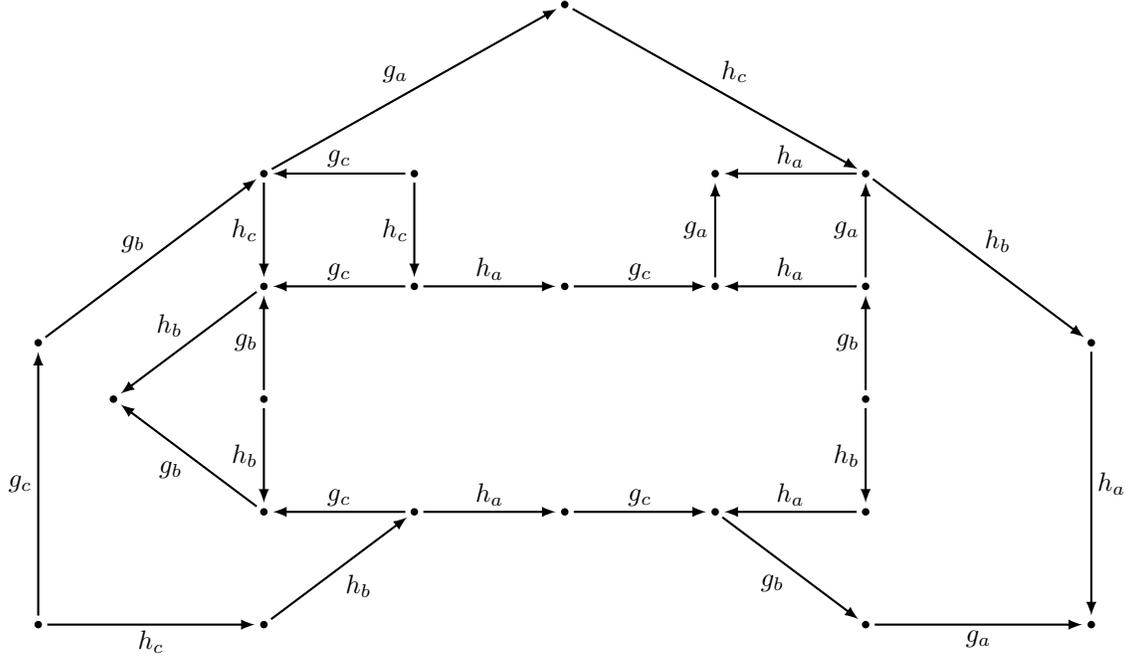

\begin{lemma}
\label{lem:decagons-commute}
Let $G$ be a group, and let
$g_1, \ldots, g_N, h_1, \ldots, h_N\in G$ satisfy
\begin{align}
g_ah_a &= h_ag_a  &&\quad \text{for all $1 \le a \le N$}, \label{e ab 1vars}\\
g_{b}g_ah_{b}h_a &= h_{b}h_ag_{b}g_a  &&\quad \text{for all $1 \leq a < b \leq N$}, \label{e ab 2vars}\\
g_b(g_c^{-1}h_ag_ch_a^{-1}) &= (g_c^{-1}h_ag_ch_a^{-1})g_b && \quad \text{for all $1 \leq a<b<c \leq N$}, \label{e ab 3vars trick a}\\
h_b^{-1}(g_c^{-1}h_ag_ch_a^{-1}) &= (g_c^{-1}h_ag_ch_a^{-1})h_b^{-1}&& \quad \text{for all $1 \leq a<b<c \leq N$}. \label{e ab 3vars trick b}
\end{align}
Then 
\begin{align}
g_N g_{N-1}\cdots g_1 h_N h_{N-1}\cdots h_1 = h_N h_{N-1}\cdots h_1 g_N g_{N-1}\cdots g_1. \label{e ab Nvars}
\end{align}
\end{lemma}

\begin{proof}
Induction on $N$.
The cases  $N = 1$ and $N=2$ are immediate from~\eqref{e ab 1vars}--\eqref{e ab 2vars}.
The case $N=3$ follows from Lemma \ref{l pairs commute group only}
since the relations \eqref{e ab 3vars trick a}--\eqref{e ab 3vars
  trick b} imply~\eqref{e ab 3vars trick 3 only}.

Now assume  $N \geq 4$.
By the inductive hypothesis, the following relations hold:
\begin{align}
g_{N-1}g_{N-2}\cdots g_2h_{N-1}h_{N-2}\cdots h_2 &=
h_{N-1}h_{N-2}\cdots h_2g_{N-1}g_{N-2}\cdots g_2. \label{e ab
  Nm12vars}\\
g_{N-1}g_{N-2}\cdots g_1h_{N-1}h_{N-2}\cdots h_1 &=
h_{N-1}h_{N-2}\cdots h_1g_{N-1}g_{N-2}\cdots g_1. \label{e ab Nm1vars}
\\
g_Ng_{N-1}\cdots g_2h_Nh_{N-1}\cdots h_2 &= h_Nh_{N-1}\cdots h_2g_Ng_{N-1}\cdots g_2. \label{e ab N2vars}
\end{align}
It remains to verify that the relations
\eqref{e ab 1vars}--\eqref{e ab 3vars trick b}
and \eqref{e ab Nm12vars}--\eqref{e ab N2vars} imply~\eqref{e ab Nvars}.
The Dehn diagram in Figure \ref{f Nvars} illustrates the argument in the
 case $N=4$.
In the diagram,
\begin{itemize}
\item
the leftmost bounded face corresponds to the relation~\eqref{e ab
  N2vars};
\item
the rightmost bounded face corresponds to the relation~\eqref{e ab
  Nm1vars};
\item
the octagonal face on the left corresponds to the relation~\eqref{e ab
  Nm12vars};
\item
the octagonal face at the top corresponds to the relation~\eqref{e ab
  2vars};
\item
the two quadrilateral faces correspond to the relation~\eqref{e ab
  1vars};
\item
the four inner
rectangles correspond to the relations~\eqref{e ab 3vars trick
  a}--\eqref{e ab 3vars trick b};
\item
and the outer face corresponds to~\eqref{e ab Nvars}.
\end{itemize}
The general case is similar,
with  $N$ rectangles in the center.
\end{proof}

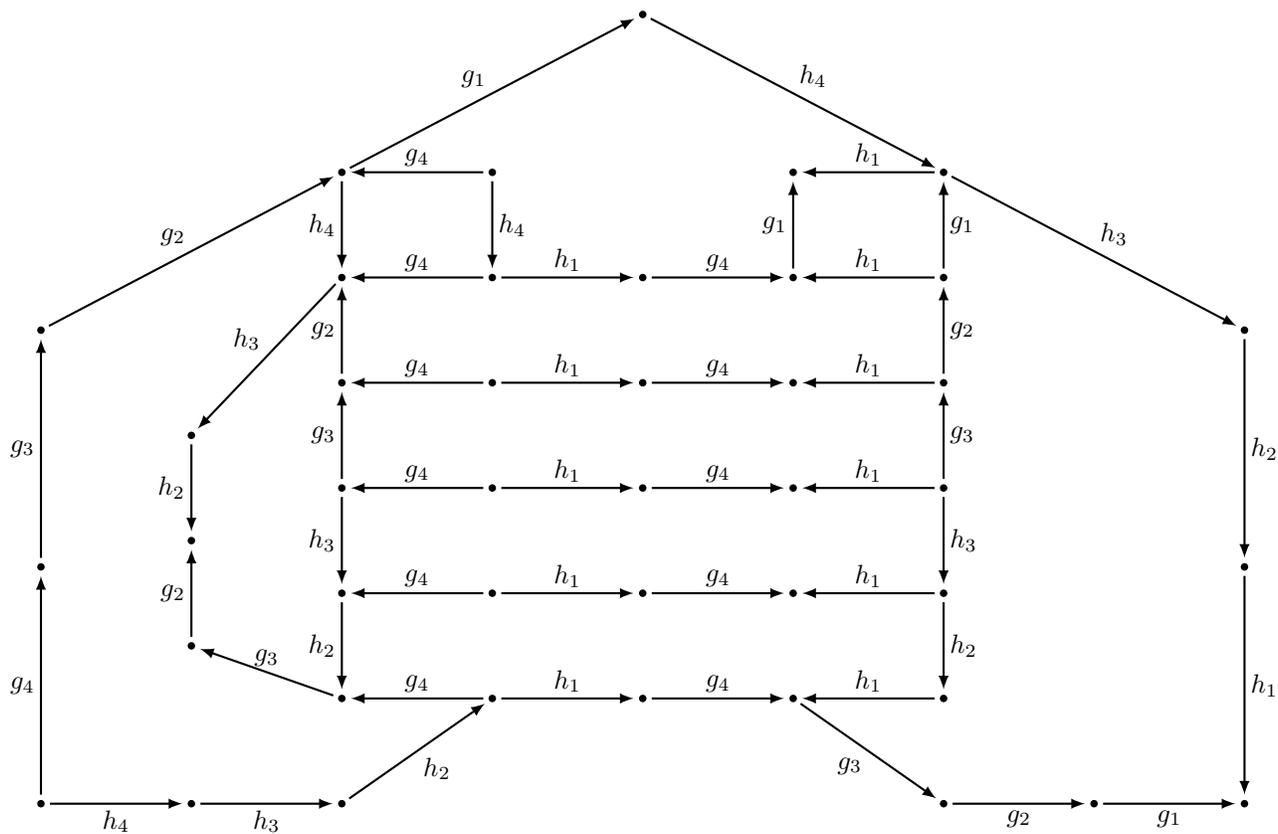
\begin{figure}[ht]
\centerfloat
\begin{tikzpicture}[xscale = 2,yscale = 1.4]
\tikzstyle{vertex}=[circle, fill, inner sep=1pt, outer sep=2pt]
\tikzstyle{framedvertex}=[inner sep=3pt, outer sep=4pt, draw=gray]
\tikzstyle{circledvertex}=[ellipse, draw, inner sep=1pt, outer sep=3pt, draw=gray]
\tikzstyle{aedge} = [draw, thick, -latex,black, inner sep = 2pt]
\tikzstyle{edge} = [draw, thick, -,black]
\tikzstyle{doubleedge} = [draw, thick, double distance=1pt, -,black]
\tikzstyle{hiddenedge} = [draw=none, thick, double distance=1pt, -,black]
\tikzstyle{dashededge} = [draw, very thick, dashed, black]
\tikzstyle{LabelStyleH} = [text=black, anchor=south, label distance = 10pt]
\tikzstyle{LabelStyleHn} = [text=black, anchor=north, label distance = 0pt]
\tikzstyle{LabelStyleV} = [text=black, anchor=east]
\tikzstyle{LabelStyleVn} = [text=black, anchor=west]
\tikzstyle{LabelStyleNE} = [text=black, anchor=north east]
\tikzstyle{LabelStyleNEn} = [text=black, anchor=south west]
\tikzstyle{LabelStyleNW} = [text=black, anchor=north west]
\tikzstyle{LabelStyleNWn} = [text=black, anchor=south east]

\node[vertex] (vb0) at (-2,0){};
\node[vertex] (va0) at (-1,0){};
\node[vertex] (v00) at (0,0){};
\node[vertex] (v10) at (1,0){};
\node[vertex] (v20) at (2,0){};

\node[vertex] (vb1) at (-2,1){};
\node[vertex] (va1) at (-1,1){};
\node[vertex] (v01) at (0,1){};
\node[vertex] (v11) at (1,1){};
\node[vertex] (v21) at (2,1){};

\node[vertex] (vb2) at (-2,2){};
\node[vertex] (va2) at (-1,2){};
\node[vertex] (v02) at (0,2){};
\node[vertex] (v12) at (1,2){};
\node[vertex] (v22) at (2,2){};

\node[vertex] (vb3) at (-2,3){};
\node[vertex] (va3) at (-1,3){};
\node[vertex] (v03) at (0,3){};
\node[vertex] (v13) at (1,3){};
\node[vertex] (v23) at (2,3){};

\node[vertex] (vb4) at (-2,4){};
\node[vertex] (va4) at (-1,4){};
\node[vertex] (v04) at (0,4){};
\node[vertex] (v14) at (1,4){};
\node[vertex] (v24) at (2,4){};

\node[vertex] (vb5) at (-2,5){};
\node[vertex] (va5) at (-1,5){};
\node[vertex] (v15) at (1,5){};
\node[vertex] (v25) at (2,5){};

\node[vertex] (v06) at (0,6.5){};

\node[vertex] (v40) at (2,-1){};
\node[vertex] (v41) at (3,-1){};
\node[vertex] (v42) at (4,-1){};
\node[vertex] (v43) at (4,1.25){};
\node[vertex] (v44) at (4,3.5){};
\node[vertex] (vd0) at (-2,-1){};
\node[vertex] (vd1) at (-3,-1){};
\node[vertex] (vd2) at (-4,-1){};
\node[vertex] (vd3) at (-4,1.25){};
\node[vertex] (vd4) at (-4,3.5){};

\node[vertex] (vc1) at (-3,.5){};
\node[vertex] (vc2) at (-3,1.5){};
\node[vertex] (vc3) at (-3,2.5){};

\draw[aedge] (v00) to node[LabelStyleH]{\footnotesize $g_4$} (v10);
\draw[aedge] (v20) to node[LabelStyleH]{\footnotesize $h_1$} (v10);
\draw[aedge] (va0) to node[LabelStyleH]{\footnotesize $g_4$} (vb0);
\draw[aedge] (va0) to node[LabelStyleH]{\footnotesize $h_1$} (v00);

\draw[aedge] (v01) to node[LabelStyleH]{\footnotesize $g_4$} (v11);
\draw[aedge] (v21) to node[LabelStyleH]{\footnotesize $h_1$} (v11);
\draw[aedge] (va1) to node[LabelStyleH]{\footnotesize $g_4$} (vb1);
\draw[aedge] (va1) to node[LabelStyleH]{\footnotesize $h_1$} (v01);

\draw[aedge] (v02) to node[LabelStyleH]{\footnotesize $g_4$} (v12);
\draw[aedge] (v22) to node[LabelStyleH]{\footnotesize $h_1$} (v12);
\draw[aedge] (va2) to node[LabelStyleH]{\footnotesize $g_4$} (vb2);
\draw[aedge] (va2) to node[LabelStyleH]{\footnotesize $h_1$} (v02);

\draw[aedge] (v03) to node[LabelStyleH]{\footnotesize $g_4$} (v13);
\draw[aedge] (v23) to node[LabelStyleH]{\footnotesize $h_1$} (v13);
\draw[aedge] (va3) to node[LabelStyleH]{\footnotesize $g_4$} (vb3);
\draw[aedge] (va3) to node[LabelStyleH]{\footnotesize $h_1$} (v03);

\draw[aedge] (v04) to node[LabelStyleH]{\footnotesize $g_4$} (v14);
\draw[aedge] (v24) to node[LabelStyleH]{\footnotesize $h_1$} (v14);
\draw[aedge] (va4) to node[LabelStyleH]{\footnotesize $g_4$} (vb4);
\draw[aedge] (va4) to node[LabelStyleH]{\footnotesize $h_1$} (v04);

\draw[aedge] (v25) to node[LabelStyleH]{\footnotesize $h_1$} (v15);
\draw[aedge] (va5) to node[LabelStyleH]{\footnotesize $g_4$} (vb5);

\draw[aedge] (vb1) to node[LabelStyleV]{\footnotesize $h_2$} (vb0);
\draw[aedge] (v21) to node[LabelStyleVn]{\footnotesize $h_2$} (v20);

\draw[aedge] (vb2) to node[LabelStyleV]{\footnotesize $h_3$} (vb1);
\draw[aedge] (v22) to node[LabelStyleVn]{\footnotesize $h_3$} (v21);

\draw[aedge] (vb2) to node[LabelStyleV]{\footnotesize $g_3$} (vb3);
\draw[aedge] (v22) to node[LabelStyleVn]{\footnotesize $g_3$} (v23);

\draw[aedge] (vb3) to node[LabelStyleV]{\footnotesize $g_2$} (vb4);
\draw[aedge] (v23) to node[LabelStyleVn]{\footnotesize $g_2$} (v24);

\draw[aedge] (vb5) to node[LabelStyleV]{\footnotesize $h_4$} (vb4);
\draw[aedge] (va5) to node[LabelStyleVn]{\footnotesize $h_4$} (va4);
\draw[aedge] (v14) to node[LabelStyleV]{\footnotesize $g_1$} (v15);
\draw[aedge] (v24) to node[LabelStyleVn]{\footnotesize $g_1$} (v25);

\draw[aedge] (vb5) to node[LabelStyleNWn]{\footnotesize $g_1$} (v06);
\draw[aedge] (v06) to node[LabelStyleNEn]{\footnotesize $h_4$} (v25);

\draw[aedge] (v25) to node[LabelStyleNEn]{\footnotesize $h_3$} (v44);
\draw[aedge] (v44) to node[LabelStyleVn]{\footnotesize $h_2$} (v43);
\draw[aedge] (v43) to node[LabelStyleVn]{\footnotesize $h_1$} (v42);
\draw[aedge] (v41) to node[LabelStyleHn]{\footnotesize $g_1$} (v42);
\draw[aedge] (v40) to node[LabelStyleHn]{\footnotesize $g_2$} (v41);
\draw[aedge, bend right=0] (v10) to node[LabelStyleNE]{\footnotesize $g_3$} (v40);

\draw[aedge] (vd4) to node[LabelStyleNWn]{\footnotesize $g_2$} (vb5);
\draw[aedge] (vd3) to node[LabelStyleV]{\footnotesize $g_3$} (vd4);
\draw[aedge] (vd2) to node[LabelStyleV]{\footnotesize $g_4$} (vd3);
\draw[aedge] (vd2) to node[LabelStyleHn]{\footnotesize $h_4$} (vd1);
\draw[aedge] (vd1) to node[LabelStyleHn]{\footnotesize $h_3$} (vd0);
\draw[aedge, bend right=0] (vd0) to node[LabelStyleNW]{\footnotesize $h_2$} (va0);

\draw[aedge] (vb4) to node[LabelStyleNWn]{\footnotesize $h_3$} (vc3);
\draw[aedge] (vc3) to node[LabelStyleV]{\footnotesize $h_2$} (vc2);
\draw[aedge] (vc1) to node[LabelStyleV]{\footnotesize $g_2$} (vc2);
\draw[aedge] (vb0) to node[LabelStyleH]{\footnotesize $g_3$} (vc1);

\end{tikzpicture}
\caption{\label{f Nvars}The proof of Lemma~\ref{lem:decagons-commute}.}
\end{figure}

\begin{proof}[Proof of Theorem~\ref{thm:group-theoretic-gi*hi}]
For $1\le a<b<c\le N$,
conditions of Theorem~\ref{thm:group-theoretic-gi*hi} include
\eqref{e ab 1vars 3 only}, \eqref{e ab 2vars 3 only}, and~\eqref{e ab
  Nvars 3 only}.
By Lemma~\ref{l pairs commute group only}, relation~\eqref{e ab 3vars
  trick 3 only} follows.
In view of condition \eqref{eq:z-commutes-with-gihi}, we then obtain
\eqref{e ab 3vars trick a}--\eqref{e ab 3vars trick b}.
Hence Lemma~\ref{lem:decagons-commute} applies, yielding~\eqref{e ab
  Nvars}, as desired.
\end{proof}

We conclude this section by slightly strengthening
Theorem~\ref{thm:group-theoretic-gi*hi},
see Corollary~\ref{cor:group-theoretic-gi*hi} below.
This will require the following immediate consequence of Lemma~\ref{l
  pairs commute group only}.

\begin{lemma}
\label{l case gb = hb}
Let $G$ be a group, and let
$g_a, g_b, g_c, h_a, h_b, h_c\in G$ satisfy
\begin{alignat*}{3}
g_ah_a &= h_ag_a,\quad &
g_b &= h_b,\quad &
g_ch_c &= h_cg_c, \\
g_{b}g_ah_{b}h_a &= h_{b}h_ag_{b}g_a,\quad&
g_{c}g_ah_{c}h_a &= h_{c}h_ag_{c}g_a,\quad&
g_{c}g_bh_{c}h_b &= h_{c}h_bg_{c}g_b.
\end{alignat*}
Then~\,$g_cg_bg_ah_ch_bh_a = h_ch_bh_ag_cg_bg_a.$
\end{lemma}

\begin{corollary}
\label{cor:group-theoretic-gi*hi}
Let $G$ be a group.
Let $g_1, \ldots, g_N, h_1, \ldots, h_N\in G$ be such that
for each $1<b<N$, either $g_b=h_b$ or condition
\eqref{eq:z-commutes-with-gihi} holds.
Then the following are equivalent:
 \begin{itemize}
\item $g_S h_S = h_S g_S$ for all subsets $S\subset\{1,\dots,N\}$;
\item $g_S h_S = h_S g_S$ for all subsets $S$ of cardinality $\le 3$.
\end{itemize}
\end{corollary}

\begin{proof}
Induction on~$N$.
For $N \le 3$, the result is clear. Now assume $N \ge 4$.
If $g_b=h_b$ for some~$b\in\{2,\dots,N-1\}$, then the claim
\[
g_N\cdots g_1 h_N\cdots h_1 = h_N\cdots h_1 g_N\cdots g_1
\]
follows by combining the induction assumption with Lemma~\ref{l case gb
  = hb}, for
$g_a = g_{b-1}\cdots g_1$, $g_c = g_{N}\cdots g_{b+1}$,
$h_a = h_{b-1}\cdots h_1$, $h_c = h_{N}\cdots h_{b+1}$.
In the only remaining case, condition~\eqref{eq:z-commutes-with-gihi}
holds for all~$b\in\{2,\dots,N-1\}$, and the claim follows from
Theorem~\ref{thm:group-theoretic-gi*hi}.
\end{proof}

\pagebreak[3]

\section{Proofs of Theorems \ref{t pairs commute v0} and \ref{th:mrot-linear-uu}}
\label{sec:proof:u=v}

We obtain Theorem \ref{t pairs commute v0} by combining
Theorem~\ref{thm:group-theoretic-gi*hi} with
Lemma~\ref{lem:lagrange-implicit} below.
While Theorem~\ref{thm:group-theoretic-gi*hi} is purely
group-theoretic,
the proof of Lemma~\ref{lem:lagrange-implicit}
implicitly relies on a Lagrange inversion argument for formal power series.

As before, we denote by $R[[x,y]]$ the
ring of formal power series in $x$ and~$y$, with
coefficients in the (unital) $\QQ$-algebra~$R$.
We use the notation $R[[x,y]]^*$ for the multiplicative subgroup of
$R[[x,y]]$ formed by power series with constant term~$1$.
Thus the elements of $R[[x,y]]$ are formal expressions
$z=\sum_{k=0}^\infty z_k$, with each $z_k\in R[x,y]$ a
homogeneous polynomial of degree~$k$ in $x$ and~$y$, with coefficients
in~$R$. (From now on, we adopt the convention $\deg(x)\!=\!\deg(y)\!=\!1$.)
We have $z\!\in\! R[[x,y]]^*$ if and only if $z_0=1\in R$.


\begin{lemma}
\label{lem:division-implicit}
Let $q\in\QQ[[x,y]]$ and $r\in R[[x,y]]$, with $q\neq 0$ and $r\neq 0$.
Then $qr\neq 0$.
\end{lemma}

\begin{proof}
Define the lexicographic order $\prec$ on the monomials~$x^i y^j$
by setting
\[
x^i y^j \prec x^{i'} y^{j'}
\stackrel{\rm def}{\Longleftrightarrow}
(\text{$i<i'$ or ($i=i'$ and $j<j'$)}).
\]
The statement of the lemma is true for $q\in\QQ$ and $r\in R$.
Consequently the leading term of the power series~$qr$,
with respect to the lexicographic order,
is the product of the leading terms of $q$ and~$r$, respectively.
The lemma follows.
\end{proof}

\begin{lemma}
\label{lem:lagrange-implicit}
Let $\varphi_1,\varphi_2,\dots\in\QQ[x,y]$ be homogeneous polynomials
of degrees $\deg(\varphi_k)=k$.
Assume that $\varphi_1\neq 0$.
Let $u\in R$, and let $f=1+\varphi_1 u +\varphi_2 u^2+\cdots\in
R[[x,y]]^*$.
If $f$ commutes with $z\in R[[x,y]]$, then $u$ commutes with~$z$.
\end{lemma}

\begin{proof}
Let $z=z_0+z_1+z_2+\cdots$, with $z_k\in R[x,y]$ a
homogeneous polynomial of degree~$k$.
Then $[f,z]=[f-1,z]=\sum_{j\ge 1}\sum_{k\ge 0} \varphi_j [u^j, z_k]$.
In order for this commutator to vanish, it must vanish in each
degree.
Since $\deg(\varphi_j [u^j, z_k])=j+k$, we conclude that
$\sum_{1 \le j \le m} \varphi_j [u^j, z_{m-j}]=0$ for every $m\ge 1$.
So we have:
\begin{itemize}
\item
$\varphi_1[u,z_0]=0$ and $\varphi_1\neq 0$, hence $[u,z_0]=0$ (by Lemma~\ref{lem:division-implicit});
\item
$\varphi_1[u,z_1]+\varphi_2[u^2,z_0]=\varphi_1[u,z_1]=0$, hence $[u,z_1]=0$;
\item
$\varphi_1[u,z_2]+\varphi_2[u^2,z_1]+\varphi_3[u^3,z_0]=\varphi_1[u,z_2]=0$, hence $[u,z_2]=0$;
\end{itemize}
and so on.
We conclude that $[u,z]=[u,z_0+z_1+z_2+\cdots]=0$, as desired.
\end{proof}

\begin{proof}[Proof of Theorem \ref{t pairs commute v0}]
Let us denote
\[
f_b = h_b^{-1}g_b
= (1-\beta_{b1} y u_b + \cdots)(1+\alpha_{b1} x u_b + \cdots)
= 1 +(\alpha_{b1} x-\beta_{b1} y) u_b +\cdots.
\]
By Theorem~\ref{thm:group-theoretic-gi*hi},
it suffices to verify that if $f_b$ commutes with $z\in R[[x,y]]^*$,
then so do both $g_b$ and~$h_b$.
Indeed, we have $\alpha_{b1} x-\beta_{b1} y\neq 0$,
so Lemma~\ref{lem:lagrange-implicit} applies;
hence $u_b$ commutes with~$z$, and therefore so do $g_b$ and~$h_b$.
\end{proof}

Replacing Theorem~\ref{thm:group-theoretic-gi*hi} by Corollary~\ref{cor:group-theoretic-gi*hi}
in the above argument, we obtain a stronger version of the Multiplicative Rule of Three:

\begin{theorem}
\label{t pairs commute v0 2}
Let $R$ be a  $\QQ$-algebra, and let $u_1,\dots,u_N\in R$.
Let $g_1, \ldots, g_N, h_1, \ldots, h_N \in R[[x,y]]^*$ be of the form
\begin{alignat}{5} \label{eq:gi-uu x or y}
g_i &= 1+\alpha_{i1} u_i &&+\alpha_{i2} u_i^2 &&+\alpha_{i3}
u_i^3&&+\cdots, \\
h_i &= 1+\beta_{i1} u_i &&+\beta_{i2} u_i^2 &&+\beta_{i3}
u_i^3&&+\cdots,\label{eq:hi-uu x or y}
\end{alignat}
where $\alpha_{ik},\beta_{ik}\in\QQ[x,y]$ are homogeneous polynomials
of degree~$k$.
Assume that for every $b\in\{2,\dots,N-1\}$,
either $g_b = h_b$ or $\alpha_{b1} \neq \beta_{b1}$.
Then the following are equivalent:
 \begin{itemize}
\item $g_S h_S = h_S g_S$ for all subsets $S\subset\{1,\dots,N\}$;
\item $g_S h_S = h_S g_S$ for all subsets $S$ of cardinality $2$
  and~$3$.
\end{itemize}
\end{theorem}

\begin{proof}[Proof of Theorem~\ref{th:mrot-linear-uu}]
It suffices to note that the assumptions in Theorem~\ref{t pairs commute v0 2}
are satisfied when $g_i = 1+\alpha_{i1} u_i$ and $h_i = 1+\beta_{i1}
u_i$,
for any linear polynomials  $\alpha_{i1}, \beta_{i1} \in \QQ[x,y]$.
\end{proof}

\section{Proof of Theorem~\ref{t half ROT}}
\label{s proof of theorem t half ROT}
\begin{lemma}
\label{l deg1 super 1}
Let $R$ be a ring, and let $g_a, g_b, v_a, v_b \in R$ satisfy
\begin{align*}
[v_a,g_a] = [v_b,g_b]&=0, 
\\
[v_b+v_a,g_b g_a] &= 0. 
\end{align*}
Then
$[v_a,g_b]g_a = g_b[g_a,v_b]$. 
\end{lemma}

\begin{proof}
This follows from the identity
\[
[v_a,g_b]g_a - g_b[g_a,v_b]
=[v_b+v_a, g_b g_a] - g_b[v_a,g_a] - [v_b,g_b]g_a\,.
\qedhere
\]
\end{proof}

\begin{lemma}
\label{l deg1 super 2}
Let $R$ be a ring, and let $g_a,g_b,g_c,v_a,v_b,v_c \in R$ satisfy
\begin{align*}
[v_b,g_b] &= 0, \\
[v_b+v_a, g_bg_a] =
[v_c+v_b, g_cg_b] &= 0, \\
[v_c+v_b+v_a,g_cg_bg_a] &= 0. 
\end{align*}
Then
$[v_a, g_c]g_bg_a = g_cg_b [g_a,v_c]$. 
\end{lemma}

\begin{proof}
This follows from the identity
\begin{align*}
&[v_a, g_c]g_bg_a - g_cg_b[g_a,v_c]\\
&=[v_c+v_b+v_a, g_cg_bg_a]
-g_c [v_b+v_a, g_bg_a]
-[v_c+v_b, g_cg_b] g_a
+g_c [v_b,g_b] g_a\,. \qedhere
\end{align*}
\end{proof}

\begin{lemma}
\label{l for super}
Let $R$ be a monoid (or a ring), and let $g_1, \ldots, g_m, z, z' \in R$ satisfy
\begin{align}
z g_{1} &= g_m z', \label{e deg1N 2}\\
z g_bg_1 &= g_m g_b z'  \qquad \text{for all $1  < b <  m$.} \label{e deg1N 3}
\end{align}
If  $g_1$ and $g_m$ are potentially invertible, then
\begin{align}
z g_{m-1}g_{m-2}\cdots g_1
=g_m g_{m-1}\cdots g_2 z'. \label{e deg1N c}
\end{align}
\end{lemma}

\begin{proof}
Passing to an extension of~$R$ wherein $g_1$ and~$g_m$ have inverses,
let us denote \mbox{$r=g_m^{-1}z=z'g_1^{-1}$} (cf.\ \eqref{e deg1N 2}).
Condition~\eqref{e deg1N 3} means that $r$ commutes with~$g_b$ for
$1\!<\!b\!<\!m$.
Hence $r$ commutes with $g_{m-1}\cdots g_2$, which is
nothing but~\eqref{e deg1N c}.
\end{proof}

\pagebreak[3]

\begin{corollary}
\label{c half ROTv2}
Let $R$ be a ring and let $v_1, \ldots, v_N, g_1, \ldots, g_N \in R$ with  $g_1,\dots,g_N$ potentially invertible.  Suppose
\begin{align*}
\text{$[\textstyle\sum_{i \in S} v_i, g_S] = 0$ for all $S \subset \{1,\ldots, N\}$ of cardinality $\le 3$.}
\end{align*}
Then for each subset $S=\{s_1<\cdots<s_m\}\subset\{1,\dots,N\}$,  
\begin{align*}
[v_{s_1}, g_{s_m}]g_{s_{m-1}}g_{s_{m-2}}\cdots g_{s_1} = g_{s_m}g_{s_{m-1}}\cdots g_{s_2}[g_{s_1},v_{s_m}].
\end{align*}
\end{corollary}
\begin{proof}
Apply Lemmas~\ref{l deg1 super 1}, \ref{l deg1 super 2}, and~\ref{l for super},
with $z=[v_{s_1},g_{s_m}]$ and $z'=[g_{s_1},v_{s_m}]$.
\end{proof}

\begin{proof}[Proof of Theorem~\ref{t half ROT}]
We need to show that relations
\begin{align}
[v_a,g_a] &= 0  \qquad \text{for all $1 \le a \le N$}; \label{e halfrot 1}\\
[v_b + v_a, g_bg_a] &= 0 \qquad \text{for all $1 \leq a < b \leq N$}; \label{e halfrot 2}\\
[v_c + v_b + v_a, g_cg_bg_a] &= 0 \qquad \text{for all $1 \leq a < b < c \leq N$} \label{e halfrot 3}
\end{align}
imply $[v_{s_m} + v_{s_{m-1}} + \cdots + v_{s_1}, g_S] = 0$ for any
subset~$S\!=\!\{s_1\!<\!\cdots\!<\!s_m\}\!\subset\!\{1,\dots,N\}$.
We establish this claim by induction on~$m=|S|$.
The cases $m\le 3$ are covered~by \eqref{e halfrot 1}--\eqref{e
  halfrot 3}.
Using the induction assumption and the Leibniz rule for commutators, we get:
\begin{align*}
& [v_{s_m} + v_{s_{m-1}} + \cdots + v_{s_1}, g_{s_m} \cdots g_{s_1}]\\
&=[(v_{s_m} + \cdots + v_{s_2}) + (v_{s_{m-1}} + \cdots + v_{s_1}) - (v_{s_{m-1}} + \cdots
    + v_{s_2}),
  g_{s_m}\cdots g_{s_1}]\\
&=[v_{s_m} + \cdots + v_{s_2}, g_{s_m}\cdots g_{s_2}] g_{s_1}
+ g_{s_m}\cdots g_{s_2} [v_{s_m} + \cdots + v_{s_2}, g_{s_1}] \\
&+ [v_{s_{m-1}}+\cdots + v_{s_1}, g_{s_m}] g_{s_{m-1}}\cdots g_{s_1}
+ g_{s_m} [v_{s_{m-1}} + \cdots + v_{s_1}, g_{s_{m-1}}\cdots g_{s_1}]
\\
&- [v_{s_{m-1}} + \cdots + v_{s_2}, g_{s_m}] g_{s_{m-1}}\cdots g_{s_1}
- g_{s_m} [v_{s_{m-1}}+\cdots+v_{s_2}, g_{s_{m-1}}\cdots g_{s_2}] g_{s_1} \\
& \qquad - g_{s_m}\cdots g_{s_2} [v_{s_{m-1}}+\cdots+v_{s_2}, g_{s_1}]\\
&=[v_{s_1}, g_{s_m}] g_{s_{m-1}}\cdots g_{s_1}-g_{s_m}\cdots g_{s_2} [g_{s_1},v_{s_m}] \\
&= 0,
\end{align*}
where the last equality is by Corollary \ref{c half ROTv2}.
\end{proof}

\begin{proof}[Proof of Theorem~\ref{th:derivations-ROT}]
This theorem is proved by exactly the same
argument as the one used for Theorem~\ref{t half ROT}.
\end{proof}


\section{Proofs of Theorems~\ref{t es commute uv 2},
\ref{t half group ROT},
\ref{t 2variable linear strengthened}, \ref{t half quadratic half group ROT}, and~\ref{t half group ROTv2}}
\label{sec:main-proofs}

\begin{lemma}
\label{lem:cba*CBA}
Let $R$ be a ring, and let $g_a,g_b,g_c,h_a,h_b,h_c\in R$, with $h_b$
potentially invertible, satisfy the relations
\begin{align}
g_bh_b&=h_bg_b,\label{eq:bB=Bb}\\
g_bg_ah_bh_a &= h_bh_ag_bg_a, \label{eq:baBA=BAba}\\
g_cg_bh_ch_b &= h_ch_bg_cg_b. \label{eq:cbCB=CBcb}
\end{align}
Then the following are equivalent:
\begin{align*}
g_cg_bg_ah_ch_bh_a &= h_ch_bh_ag_cg_bg_a,\\
g_cg_b[g_a,h_c]h_bh_a &= h_ch_b[h_a,g_c]g_bg_a.
\end{align*}
\end{lemma}

\begin{proof}
The statement follows from the identity (in the appropriate extension
of~$R$):
\begin{align}
&[g_cg_bg_a,h_ch_bh_a]=g_cg_b[g_a,h_c]h_bh_a-h_ch_b[h_a,g_c]g_bg_a \label{eq:[cba,CBA]}
\\
&\qquad +[g_cg_b,h_ch_b]h_b^{-1}g_ah_bh_a+h_ch_bg_ch_b^{-1}([g_bg_a,h_bh_a]-[g_b,h_b]h_b^{-1}g_ah_bh_a). \notag
\qedhere
\end{align}
\end{proof}

\begin{lemma}
\label{lem:strong-commute-abc}
Let $R$ be a ring,
and let $g_a,g_b,g_c,h_a,h_b,h_c$ be
potentially invertible elements of~$R$ satisfying
\eqref{eq:bB=Bb}--
\eqref{eq:cbCB=CBcb}.
Then any two of the following conditions imply the third:
\begin{align*}
g_cg_bg_ah_ch_bh_a &= h_ch_bh_ag_cg_bg_a,\\
g_cg_b[g_a,h_c] &= [h_a,g_c]g_bg_a,\\
h_ch_b[h_a,g_c] &=[h_a,g_c]h_bh_a.
\end{align*}
\end{lemma}

\begin{proof}
Adding the trivial identity
\[
0 = -[h_a,g_c]g_bg_ah_bh_a + [h_a,g_c]h_bh_ag_bg_a +[h_a,g_c][g_bg_a,h_bh_a]
\]
to \eqref{eq:[cba,CBA]}, we obtain
\begin{align*}
&[g_cg_bg_a,h_ch_bh_a]=(g_cg_b[g_a,h_c]-[h_a,g_c]g_bg_a)h_bh_a
-(h_ch_b[h_a,g_c]-[h_a,g_c]h_bh_a)g_bg_a\\
&\!+\![h_a,g_c][g_bg_a,h_bh_a]
\!+\![g_cg_b,h_ch_b]h_b^{-1}g_ah_bh_a
\!+\!h_ch_bg_ch_b^{-1}([g_bg_a,h_bh_a]\!-\![g_b,h_b]h_b^{-1}g_ah_bh_a),
\end{align*}
and the claim follows.
\end{proof}

\begin{lemma}
\label{lem:strong-commute-ab}
Let $R$ be a (unital) ring, and let $g_a,g_b,h_a,h_b$ be potentially
invertible elements of~$R$ satisfying
$g_ah_a=h_ag_a$ and $g_bh_b=h_bg_b$.
Then any two of the following conditions imply the third:
\begin{align*}
g_bg_ah_bh_a &= h_bh_ag_bg_a,\\
g_b[g_a,h_b] &= [h_a,g_b]g_a,\\
h_b[h_a,g_b] &=[h_a,g_b]h_a.
\end{align*}
\end{lemma}

\begin{proof}
This is the $g_b=h_b=1$ case of Lemma~\ref{lem:strong-commute-abc},
with a suitable change of notation.
\end{proof}

Theorem~\ref{th:master-criterion} below, although
  not a ``Rule of Three,'' is a powerful result
which will be used to prove Theorems~\ref{t half group ROT} and~\ref{t
  2variable linear strengthened}.

\begin{theorem}
\label{th:master-criterion}
Let $g_1, \ldots, g_N, h_1, \ldots, h_N$
be potentially invertible elements of a ring~$R$ satisfying the relations
\begin{align}
g_a h_a &= h_a g_a &&\text{for all $1 \leq a  \leq N$}; \label{e gh}\\
g_b[g_a,h_b] &= [h_a,g_b]g_a &&\text{for all $1 \leq a < b \leq N$}; \label{e ggh}\\
g_cg_b[g_a,h_c] &= [h_a, g_c]g_bg_a &&\text{for all $1 \leq a < b < c \leq N$}; \label{e gggh}\\
h_b[h_a,g_b] &= [h_a,g_b]h_a &&\text{for all $1 \leq a < b \leq N$}; \label{e hhg=hgh}\\
h_ch_b[h_a,g_c] &= [h_a, g_c]h_bh_a &&\text{for all $1 \leq a < b < c
  \leq N$}. \label{e hhhg=hghh}
\end{align}
Then $g_Sh_S=h_Sg_S$ for any $S \subset \{1,\ldots, N\}$.
\end{theorem}

\pagebreak[3]

\begin{proof}
First, we claim that for any subset $S\!=\!\{s_1<\cdots<s_m\}\subset\{1,\dots,N\}$,
one~has
\begin{equation}
\label{eq:[hg]ggg}
[h_{s_1}, g_{s_m}]g_{s_{m-1}}g_{s_{m-2}}\cdots g_{s_1}
=g_{s_m}g_{s_{m-1}}\cdots g_{s_2}[g_{s_1},h_{s_m}].
\end{equation}
This follows by applying Lemma~\ref{l for super}
with $z=[h_{s_1},g_{s_m}]$ and $z'=[g_{s_1},h_{s_m}]$.
(Conditions \eqref{e deg1N 2}--\eqref{e deg1N 3} hold by \eqref{e ggh}--\eqref{e gggh}.)
Again applying Lemma~\ref{l for super}, this time
with $z=z'=[h_{s_1},g_{s_m}]$
(and relying on \eqref{e hhg=hgh}--\eqref{e hhhg=hghh}),
we get
\begin{equation}
\label{eq:[hg]hhh}
[h_{s_1}, g_{s_m}]h_{s_{m-1}}h_{s_{m-2}}\cdots h_{s_1}
=h_{s_m}h_{s_{m-1}}\cdots h_{s_2}[h_{s_1},g_{s_m}].
\end{equation}

We now prove $g_Sh_S=h_Sg_S$ by induction
on~$m=|S|$.
The base case $m=1$ is given in \eqref{e gh}.
It remains to invoke Lemma~\ref{lem:strong-commute-abc}
with
\begin{alignat*}{3}
g_a &= g_{s_1}\,,\quad & g_b &= g_{s_{m-1}}\cdots g_{s_2}\,,\quad &
g_c &= g_{s_m}\,, \\
h_a &= h_{s_1}\,,\quad & h_b &=h_{s_{m-1}}\cdots h_{s_2}\,,\quad &
h_c &= h_{s_m}\,,
\end{alignat*}
making use of \eqref{eq:[hg]ggg} and~\eqref{eq:[hg]hhh}.
\end{proof}

\begin{proof}[Proof of Theorem~\ref{t half group ROT}]
We need to show that relations
\begin{align}
g_a h_a &= h_a g_a &&\text{for all $1 \leq a  \leq N$}; \label{eq:gh=hg} \\
g_bg_a h_bh_a &= h_bh_a g_bg_a &&\text{for all $1 \leq a < b \leq N$}; \label{eq:gghh}\\
g_cg_bg_a h_ch_bh_a &= h_ch_bh_a g_cg_bg_a &&\text{for all $1 \leq a <
  b < c \leq N$};\label{eq:ggghhh}\\
g_bg_a(h_b + h_a) &= (h_b + h_a) g_bg_a && \text{for all $1 \leq a < b \leq N$}; \label{e linear 2}\\
g_cg_bg_a(h_c + h_b + h_a) &=  (h_c + h_b + h_a) g_cg_bg_a&& \text{for all $1 \leq a < b < c \leq N$} \label{e linear 3}
\end{align}
imply $g_S h_S = h_S g_S$ for all subsets~$S$.
(The other conclusion is by Theorem~\ref{t half ROT}.)
By Theorem~\ref{th:master-criterion},
the claim will follow once we have checked conditions \eqref{e gh}--\eqref{e hhhg=hghh}.

Relation~\eqref{eq:gh=hg} is the same as~\eqref{e gh}.
By Lemma~\ref{l deg1 super 1}, relations \eqref{eq:gh=hg} and~\eqref{e linear 2}
imply~\eqref{e ggh}.
By Lemma~\ref{l deg1 super 2}, relations \eqref{eq:gh=hg} and \eqref{e
  linear 2}--\eqref{e linear 3}
imply~\eqref{e gggh}.
Finally, by Lemmas~\ref{lem:strong-commute-abc}--\ref{lem:strong-commute-ab},
relations \eqref{e ggh}--\eqref{e gggh} and \eqref{eq:gh=hg}--\eqref{eq:ggghhh}
imply \eqref{e hhg=hgh}--\eqref{e hhhg=hghh}.
\end{proof}

\begin{corollary}
\label{c 2variable linear strengthened v2}
Let $g_1, \ldots, g_N, h_1, \ldots, h_N$ be potentially invertible
elements of a ring $R$ satisfying
\begin{align}
[g_a, h_b]&=[h_a,g_b]    && \text{for all $1 \le a < b \le
  N$}; \label{e switch v2}
\\
[g_a,h_a] &= 0 &&\text{for all $1 \leq a  \leq N$}; \label{e gh'}
\\
g_b[g_a,h_b] &= [h_a,g_b]g_a &&\text{for all $1 \leq a < b \leq N$};
\label{e ggh'}
\\
g_cg_b[g_a,h_c] &= [h_a, g_c]g_bg_a &&\text{for all $1 \leq a < b < c
  \leq N$};
\label{e gggh'}
\\
h_b[h_a,g_b] &= [g_a,h_b]h_a &&\text{for all $1 \leq a < b \leq N$}; \label{e hhg}\\
h_ch_b[h_a,g_c] &= [g_a, h_c]h_bh_a &&\text{for all $1 \leq a < b < c \leq N$}. \label{e hhhg}
\end{align}
Then $g_Sh_S=h_Sg_S$ for any $S \subset \{1,\ldots, N\}$.
\end{corollary}

\begin{proof}
Substitute~\eqref{e switch v2} into
\eqref{e hhg}--\eqref{e hhhg} to get
\eqref{e hhg=hgh}--\eqref{e hhhg=hghh};
then apply Theorem~\ref{th:master-criterion}.
\end{proof}

\begin{proof}[Proof of Theorem~\ref{t 2variable linear strengthened}]
We will use Corollary~\ref{c 2variable linear strengthened v2} to show
that \eqref{e linear 1}--\eqref{e linear 3b}
imply  $g_Sh_S= h_Sg_S$ for all subsets $S$.
(The other conclusions are by Theorem~\ref{t half ROT} and Remark~\ref{rem:add-rule-of-two}.)
Relations \eqref{e switch v2}--\eqref{e gh'}
are equivalent to \eqref{e linear 1}--\eqref{e linear 2sum}.
Relations \eqref{e ggh'}--\eqref{e gggh'}
(which are identical to \eqref{e ggh}--\eqref{e gggh})
are checked precisely as in the proof of
Theorem~\ref{t half group ROT}, using \eqref{e linear 22}--\eqref{e
  linear 33} and Lemmas~\ref{l deg1 super 1}--\ref{l deg1 super 2}.
In~the same way, we use \eqref{e linear 2b}--\eqref{e linear 3b}
to obtain \eqref{e hhg}--\eqref{e hhhg}.
\end{proof}

\begin{proof}[Proof of Theorem \ref{t es commute uv 2}]
Set $g_i = 1+xu_i$ and $h_i = 1+yv_i$ for $i = 1,\ldots, N$.
(Here, as before, we are operating in the ring of formal power series
in two variables $x$ and~$y$.)
Then the $g_i$ and the $h_i$ are invertible.
Furthermore, the relations \eqref{e uv 12vars AB}--\eqref{e uvvv
  23vars AB} imply the relations \eqref{e linear 1}--\eqref{e linear
  3b}.
Applying Theorem \ref{t 2variable linear strengthened},
we conclude that $g_Sh_S=h_Sg_S$
for any $S \subset \{1,\ldots, N\}$.
Equivalently,
$e_k(\mathbf{u}_S)e_\ell(\mathbf{v}_S)=e_\ell(\mathbf{v}_S)e_k(\mathbf{u}_S)$
for all $k$, $\ell$, as desired.
\end{proof}

\begin{proof}[Proof of Theorem~\ref{t half group ROTv2}]
Set  $h_i = 1+yf_i\in R[y]$;
here, as before,  $y$~is a formal variable commuting with all elements of  $R$.
One then checks that the two statements in Theorem~\ref{t half group ROT}
translate into the respective
statements in Theorem~\ref{t half group ROTv2}. The claim follows.
\end{proof}

\begin{proof}[Proof of Theorem \ref{t half quadratic half group ROT}]
Same argument as above, 
this time with
$h_i = 1+yf_i \in R[y]/(y^3)$.
\end{proof}

We conclude this section with additional results on the Multiplicative
Rule of Three, cf.\ Definition~\ref{def:mult-rot}.

\begin{corollary}
\label{c half group ROT}
The Multiplicative Rule of Three holds for
$g_i = 1+\alpha_{i1} x +\alpha_{i2} x^2 +\cdots$
and
$h_i = 1+yv_i$,
for any $\alpha_{ik}\in \mathcal{A}=\QQ\langle u_1, \ldots, u_M, v_1, \ldots, v_M\rangle$.
\end{corollary}

\begin{proof}
Let $I$ be an ideal in~$\mathcal{A}$, and let $I[[x,y]]$ be the
ideal generated by $I$ inside~$\mathcal{A}[[x,y]]$.
Suppose $[g_S, h_S] \equiv 0\bmod I[[x,y]]$ for all $|S| \!\le\! 3$.
Taking the coefficient of~$y$, we get $[g_S, \sum_{i \in S}v_i] \equiv 0\bmod I[[x,y]]$
and consequently $[g_S, \sum_{i \in S}h_i] \equiv 0\bmod I[[x,y]]$.
It remains to apply Theorem~\ref{t half group ROT} (with $R=\mathcal{A}[[x,y]]/I[[x,y]]$).
\end{proof}

One can more generally identify specific conditions on the
$\beta_{ij}$ in Conjecture \ref{cj pairs commute uv}
which ensure that
$g_S (\sum_{i \in S}h_i) = (\sum_{i \in S}h_i) g_S$ for all
$|S| \le 3$.  Here is one example.

\begin{corollary}
\label{c cj pairs commute uv version3}
The Multiplicative Rule of Three holds for
\mbox{$g_i = 1+\alpha_{i1} x +\alpha_{i2} x^2 +\cdots$} and
$h_i = 1+ \beta_{i1} yv_i + \beta_{i4}y^4v_i^4$, for any $\alpha_{ik}\in \mathcal{A}$ and $\beta_{i1}, \beta_{i4} \in \QQ$.
\end{corollary}
\begin{proof}
Same argument as above, this time noting that for $d = 1,4$ and  $|S|
\le 3$, taking the coefficient of  $y^d$ in
$[g_S,h_S] \equiv 0$ yields $[g_S, \sum_{i \in S}\beta_{id}v_i^d] \equiv 0$.
\end{proof}

\pagebreak[3]

\section{Proof of Theorem~\ref{t super}}
\label{sec:proof-super}

%

We will need the following slight generalization of
Lemma~\ref{lem:strong-commute-abc}.

\begin{lemma}
\label{lem:strong-commute-abc-with-z}
Let $R$ be a ring, and let $g_a,g_b,g_c,h_a,h_b,h_c$ be
potentially invertible elements of~$R$ satisfying
\eqref{eq:bB=Bb}--
\eqref{eq:cbCB=CBcb}.
Let $z\in R$.
Then any two of the following conditions imply the third:
\begin{align*}
g_cg_bg_ah_ch_bh_a &= h_ch_bh_ag_cg_bg_a,\\
g_cg_b[g_a,h_c] &= z g_bg_a,\\
h_ch_b[h_a,g_c] &=z h_bh_a.
\end{align*}
\end{lemma}

\begin{proof}
This follows from a modified version of the identity in the proof
of Lemma~\ref{lem:strong-commute-abc}:
\begin{align*}
&[g_cg_bg_a,h_ch_bh_a]=(g_cg_b[g_a,h_c]-z g_bg_a)h_bh_a
-(h_ch_b[h_a,g_c]-z h_bh_a)g_bg_a\\
&\!+\!z[g_bg_a,h_bh_a]
\!+\![g_cg_b,h_ch_b]h_b^{-1}g_ah_bh_a
\!+\!h_ch_bg_ch_b^{-1}([g_bg_a,h_bh_a]\!-\![g_b,h_b]h_b^{-1}g_ah_bh_a).
\qedhere
\end{align*}
\end{proof}


\begin{lemma}
\label{l super gen fun}
Let $A$ be a ring, and let $u_1, \ldots, u_N, v_1, \ldots,
v_N \in A$ satisfy
\begin{equation}
[u_a,v_b] = [v_a,u_b]  \quad \text{for all $1 \le a < b  \le
  N$}.  \label{e lin super 2sum}
\end{equation}
For every $i\in\{1,\dots,N\}$, let $\alpha_i,\beta_i$ be central elements of  $A$
such that $1 + \alpha_i u_i$ and
$1 + \beta_i v_i$ are invertible, and define $g_i,h_i\in A$
by making one of the following two choices:
\begin{equation}
\label{eq:gi-hi-cases}
\text{either\ \ }
\begin{cases}
g_i = 1 + \alpha_i u_i\,,\\
h_i = 1 + \beta_i v_i\,,
\end{cases}
\text{or\ \ }
\begin{cases}
g_i = (1 + \alpha_i u_i)^{-1}\,,\\
h_i = (1 + \beta_i v_i)^{-1}\,.
\end{cases}
\end{equation}
Suppose that the following relations are satisfied:
\begin{align}
[v_a,g_a] &= 0  \qquad \text{for all $1 \le a \le N$}; \label{e lin super 1}\\
[v_b + v_a, g_bg_a] &= 0 \qquad \text{for all $1 \leq a < b \leq N$}; \label{e lin super 2}\\
[v_c + v_b + v_a, g_cg_bg_a] &= 0 \qquad \text{for all $1 \leq a < b < c \leq N$}; \label{e lin super 3} \\
[u_a,h_a] &= 0  \qquad \text{for all $1 \le a \le N$}; \label{e lin super 1hp}\\
[u_b + u_a, h_bh_a] &= 0 \qquad \text{for all $1 \leq a < b \leq N$}; \label{e lin super 2b}\\
[u_c + u_b + u_a, h_ch_bh_a] &= 0 \qquad \text{for all $1 \leq a < b < c \leq N$}. \label{e lin super 3b}
\end{align}
Then
$g_N g_{N-1} \cdots g_1\, h_N h_{N-1}\cdots h_1
= h_N h_{N-1} \cdots h_1 \,g_N g_{N-1} \cdots g_1$.
\end{lemma}

Lemma \ref{l super gen fun} can be thought of as a fancy version of Theorem \ref{t
  2variable linear strengthened}.

\begin{proof}
Let  $P(m,n)$ be the statement of the lemma with
$1$ and~$N$ replaced by $m$ and~$n$, respectively.
We will prove  $P(m,n)$ by induction on $n-m$.
The case $n=m$ is the relation
\begin{equation}
[g_a,h_a] = 0 \quad \text{for all $1 \le a \le N$}, \label{e lin super
  1gh}
\end{equation}
which is immediate from~\eqref{eq:gi-hi-cases} and~\eqref{e lin super 1}
(or~\eqref{e lin super 1hp}).
Now assume $m<n$.

First notice that  $P(m,n)$ is equivalent to  $P(m,n)$ with
$g_{m+n-i}^{-1}$ in place of~$g_i$,
$h_{m+n-i}^{-1}$ in place of~$h_i$,
and $u_{m+n-i}$ (resp., $v_{m+n-i}$, $\alpha_{m+n-i}$,~$\beta_{m+n-i}$)
in place of~$u_i$ (resp.,~$v_i$, $\alpha_i$,~$\beta_i$).
For instance, $[v_c + v_b + v_a, g_cg_bg_a] = 0$ is
equivalent to $[v_c + v_b + v_a, g_a^{-1}g_b^{-1}g_c^{-1}] = 0$;
the set of these relations over all $m \leq a < b < c \leq n$ is
identical to the relations $[v_{n+m-c} + v_{n+m-b} + v_{n+m-a},
  g_{n+m-c}^{-1}g_{n+m-b}^{-1}g_{n+m-a}^{-1}] = 0$
over all $m \leq a < b < c \leq n$.

By the previous paragraph, we may assume that at least one of the following holds:
\begin{align}
&\text{$g_m = 1 + \alpha_m u_m$ and  $h_m = 1 + \beta_m v_m$}; \label{e m unbarred}\\
&\text{$g_n = 1 + \alpha_n u_n$ and  $h_n = 1 + \beta_n v_n$.} \label{e n unbarred}
\end{align}
Let us assume~\eqref{e n unbarred}, as the case of~\eqref{e
  m unbarred}
is handled similarly.
By \eqref{e lin super 2sum} and~\eqref{e n unbarred}, we have
\begin{align}
\alpha_n [u_m,  h_n] = \alpha_n \beta_n [u_m,  v_n] = \alpha_n \beta_n [v_m,  u_n] = \beta_n[v_m,  g_n].
\label{eq:[dm,hn]}
\end{align}
By Corollary \ref{c half ROTv2}, the relations \eqref{e lin super
  1}--\eqref{e lin super 3b} imply
\begin{align}
g_ng_{n-1}\cdots g_{m+1}[g_m,v_n] = [v_m, g_n]g_{n-1}g_{n-2}\cdots g_m,  \label{e halfrot cv4}\\
h_nh_{n-1}\cdots h_{m+1}[h_m,u_n] = [u_m, h_n]h_{n-1}h_{n-2}\cdots h_{m}. \label{e halfrot cv5}
\end{align}
Next multiply \eqref{e halfrot cv4}, \eqref{e halfrot cv5} by  $\beta_n$, $\alpha_n$, respectively,
and apply \eqref{e n unbarred} and \eqref{eq:[dm,hn]} to obtain
\begin{align}
g_ng_{n-1}\cdots g_{m+1}[g_m,h_n] = \beta_n[v_m, g_n]g_{n-1}g_{n-2}\cdots g_m,  \label{e halfrot cv6}\\
h_nh_{n-1}\cdots h_{m+1}[h_m,g_n] = \beta_n[v_m, g_n]h_{n-1}h_{n-2}\cdots h_{m}. \label{e halfrot cv7}
\end{align}
Now $P(m,n)$ follows by applying Lemma~\ref{lem:strong-commute-abc-with-z}
with $g_a = g_{m}$, $g_b = g_{n-1}\cdots g_{m+1}$, $g_c = g_{n}$,
$h_a = h_{m}$, $h_b =h_{n-1}\cdots h_{m+1}$, $h_c = h_{n}$,  $z = \beta_n[v_m, g_n]$,
making use of \eqref{e halfrot cv6}--\eqref{e halfrot cv7} and the
induction assumption.
\end{proof}

\begin{proof}[Proof of Theorem \ref{t super}]
The statement is proved by applying
Lemma~\ref{l super gen fun} with $A = R[[x,y]]$ and $g_i, h_i$ given by
\begin{align} g_i= \begin{cases}
1+xu_i & \text{ if $i$ is unbarred}\\
(1-xu_i)^{-1} & \text{ if $i$ is barred,}
\end{cases}
\qquad h_i= \begin{cases}
1+yv_i & \text{ if $i$ is unbarred}\\
(1-yv_i)^{-1} & \text{ if $i$ is barred.}
\end{cases} \label{e gh super def} \end{align}
(Thus $\alpha_i = x, \beta_i = y$ if $i$ is unbarred and $\alpha_i = -x, \beta_i = -y$ if $i$ is barred.)

We verify the relation \eqref{e lin super 2sum} using the $|S| \le 2$,
$k = 1$ cases of~\eqref{e 1d super}:
\[
[u_a,v_b] - [v_a,u_b] = [ u_b +  u_a,
  v_b +  v_a] - [u_a,v_a] - [u_b,v_b]= 0.
\]
The remaining relations \eqref{e lin super 1}--\eqref{e lin super 3b}
follow from \eqref{e 1d super}--\eqref{e 1d super2}, using the fact that
\begin{align*}
&g_{S} =\textstyle\sum_{k}x^k\, \esup{k}{\mathbf{u}_S},\\
&h_{S} =\textstyle\sum_{\ell}y^\ell\, \esup{\ell}{\mathbf{v}_S},
\end{align*}
for any $S \subset \{1,\ldots, N\}$.
So Lemma~\ref{l super gen fun} applies,
and Theorem \ref{t super} follows.
\end{proof}

\noindent
\textbf{Acknowledgments.}
We thank the participants of the conference
\emph{Lie and Jordan Algebras VI} (Bento Gon\c{c}alves, December 2015) for
stimulating discussions.
We thank Elaine~So for help typing and typesetting figures and the anonymous referee for a number of helpful comments.

\end{document}